\documentclass[aos,preprint]{imsart}

\RequirePackage[OT1]{fontenc}
\RequirePackage{amsthm,amsmath}
\RequirePackage[]{natbib}
\RequirePackage[colorlinks,citecolor=blue,urlcolor=blue]{hyperref}

\usepackage[ansinew]{inputenc}
\usepackage{hhline}
\usepackage{framed} 
\usepackage{graphicx}
\usepackage{color}
\usepackage{amsmath}
\usepackage{amssymb}
\usepackage{amsthm} 
\usepackage[geometry]{ifsym}
\usepackage{dsfont}
\usepackage{amsfonts}
\usepackage{accents}

\renewcommand{\S}{\mathcal S}
\newcommand{\G}{\mathcal G}

\newcommand{\R}{\mathds R}
\newcommand{\ind}{\mathds 1}
\newcommand{\N}{\mathds N}

\newcommand{\E}{\mathds E}
\renewcommand{\P}{\mathds P}
\newcommand{\Q}{\mathds Q}

\newcommand{\widesim}[2][1.5]{\mathrel{\overset{#2}{\scalebox{#1}[1]{$\sim$}}}}

\DeclareMathOperator{\Cov}{Cov}	

\DeclareMathOperator{\rank}{rank}

\DeclareMathOperator{\Var}{Var}

\DeclareMathOperator{\Id}{Id}
\let\Im\relax
\DeclareMathOperator{\Im}{Im}

\DeclareMathOperator{\sign}{sign}

\startlocaldefs
\numberwithin{equation}{section}
\theoremstyle{plain}
\newtheorem{theorem}{Theorem}[section]
\newtheorem{proposition}{Proposition}[section]
\newtheorem{lemma}{Lemma}[section]
\newtheorem{corollary}{Corollary}[section]
\newtheorem{remark}{Remark}[section]
\endlocaldefs

\begin{document}

\begin{frontmatter}
\title{The critical threshold level on Kendall's tau statistic concerning minimax estimation of sparse correlation matrices}
\runtitle{The critical threshold level}
\thankstext{T1}{Supported by the DFG research unit FG 1735, RO 3766/3-1.}

\begin{aug}
\author{\fnms{Kamil} \snm{Jurczak}\thanksref{T1}\ead[label=e1]{kamil.jurczak@ruhr-uni-bochum.de}}

\runauthor{Kamil Jurczak}

\affiliation{Ruhr-Universit\"at Bochum}

\address{Kamil Jurczak\\
Universit\"atsstrasse 150\\
44801 Bochum\\
Germany\\
\printead{e1}\\
\phantom{Email:\ }}
%\address{Address of the Third author\\
%Usually a few lines long\\
%Usually a few lines long\\
%\printead{e3}\\
%\printead{u1}}
\end{aug}

\begin{abstract} In a sparse high-dimensional elliptical model we consider a hard threshold estimator for the correlation matrix based on Kendall's tau with threshold level $\alpha(\frac{\log p}{n})^{1/2}$. Parameters $\alpha$ are identified such that the threshold estimator achieves the minimax rate under the squared Frobenius norm and the squared spectral norm. This allows a reasonable calibration of the estimator without any quantitative information about the tails of the underlying distribution. %where the parameters $c_{n,p}$ and $q$ depend on the class of sparse correlation matrices.\\
For Gaussian observations we even establish a critical threshold constant $\alpha^\ast$ under the squared Frobenius norm, i.e. the proposed estimator attains the minimax rate for $\alpha>\alpha^\ast$ but in general not for $\alpha<\alpha^\ast$. 
%This critical value $\alpha^\ast$ is given by $\frac{\sqrt{2}\pi}{3}$. 
To the best of the author's knowledge this is the first work concerning critical threshold constants. The main ingredient to provide the critical threshold level is a sharp large deviation expansion for Kendall's tau sample correlation evolved from an asymptotic expansion of the number of permutations with a certain number of inversions.

The investigation of this paper also covers further statistical problems like the estimation of the latent correlation matrix in the transelliptical and nonparanormal family.
\end{abstract}



\end{frontmatter}

\section{Introduction}\label{sec:int}
\noindent  Let $X_1,...,X_n$ be $n$ i.i.d. observations in $\R^p$ with population covariance matrix $\Sigma$ and correlation matrix $\rho$. The estimation of $\Sigma$ and $\rho$ is of great interest in multivariate analysis, among others for principal component analysis and linear discriminant analysis, see also \cite{Bickel2008b} and \cite{Pourahmadi2013} for further applications. In modern statistical problems like for example in gene expression arrays  the dimension $p$ of the observations is typically much larger than the sample size $n$. In this case, the sample covariance matrix is a poor and in general inconsistent estimator for the population covariance matrix (see for example \cite{Johnstone2001} and \cite{Baik2006}). Therefore the problem of estimating high-dimensional covariance matrices has been investigated under a variety of additional structural assumptions on $\Sigma$. For instance, the spiked covariance model, i.e. the population covariance matrix has a representation as the sum of a diagonal and a low-rank matrix, with sparse leading eigenvectors - and variants of it - has been extensively studied (cf. e.g. \cite{Johnstone2009}, \cite{Johnstone2007}, \cite{daspremont2007}, \cite{Shen2008}, \cite{Amini2009}, \cite{Cai2012}, \cite{Fan2013}, \cite{Berthet2013}, \cite{Vu2013}  and \cite{Caid} among others).

A further common assumption is sparsity on the covariance matrix itself (cf. e.g. \cite{Bickel2008}, \cite{Karoui2008} and \cite{Levina2012}). While in some models there is more information about the structure of the sparsity as for band covariance matrices, the case, that the position of the entries with small magnitude is unknown, is of particularly great interest. One possibility to model the latter assumption is to assume that each row lies in a (weak) $\ell_q$-ball (see \cite{Bickel2008b} and \cite{Caib}). While \cite{Bickel2008b} proved the consistency of an entrywise threshold estimator based on the sample covariance if $\frac{\log p}{n}$ tends to zero, \cite{Caib} established the minimax rates of estimation for $\Sigma$ under the assumption that each row is an element of a weak $\ell_q$-ball and proved that the threshold estimator proposed by \cite{Bickel2008b} attains the minimax rate if the constant $\alpha>0$ in the threshold level $\alpha(\frac{\log p}{n})^{1/2}$ is sufficiently large. Their results hold for subgaussian random vectors $X_1,...,X_n$ and therefore $\alpha$ depends on the largest $\Psi$-Orlicz norm of all  $2$-dimensional subvectors of $X_1$, where $\Psi(z)=\exp(z^2)-1$. Hence, the threshold estimator requires quantitative prior information about the subgaussian tails. For general classes of distributions on $\R^p$ (with a least two moments) the situation is even more intricate and in a way hopeless. The choice of an effective threshold level inherently needs precise knowledge of the tail behavior of the underlying distribution (\cite{Karoui2008}). Typically, this information is not available in advance.

On the contrary, in elliptical models there exists at least an estimator for the entries of the correlation matrix with a universal tail behavior. The foundation of this observation is the fundamental relation $\sin(\frac{\pi}{2}\tau_{ij})=\rho_{ij}$ between Kendall's tau $\tau_{ij}$ and Pearson's correlation $\rho_{ij}$ of two non-atomic components in an elliptical random vector  (cf. \cite{Hult2002}). Then $\hat\rho_{ij}=\sin(\frac{\pi}{2}\hat\tau_{ij})$ based on Kendell's tau sample correlation $\hat\tau_{ij}$ is a natural estimator for $\rho_{ij}$. By Hoeffding's inequality for U-statistics $\hat\tau_{ij}$ and therefore also $\hat\rho_{ij}$ have quantifiable subgaussian tails. This enables to identify reasonable threshold levels $\alpha(\frac{\log p}{n})^{1/2}$ for a threshold estimator based on Kendall's tau. Thereto, we consider the following model. Let  $X_1,...,X_n\in\R^p$ be an i.i.d. sample from an elliptical distribution with correlation matrix $\rho$ and Kendall's tau correlation matrix $\tau$ such that the components $X_{i1}$, $i=1,...,p$, have a non-atomic distribution. Under the unknown row-wise sparsity condition on $\rho$ we seek the least sparse estimator among the family of entrywise threshold estimators $\{\hat\rho^\ast_\alpha|\alpha>0\}$ attaining the optimal speed of convergence, where $\hat\rho_\alpha^\ast=:\hat\rho^\ast=(\hat\rho^\ast_{ij})$ consists of the entries \[\hat\rho_{ij}^\ast:=\sin\left(\frac{\pi}{2}\hat\tau_{ij}\right)\ind\left(|\hat\tau_{ij}|>\alpha\sqrt{\frac{\log p}{n}}\right) \text{ for }i\neq j\text{ and }\hat\rho^\ast_{ii}=1.\]

It is shown that $\hat\rho^\ast$ achieves the minimax rate $c_{n,p}(\frac{\log p}{n})^{1-q/2}$ under the squared Frobenius norm loss if $\alpha>2$, where the parameters $c_{n,p}$ and $q$ depend on the class of sparse correlation matrices. Under the squared spectral norm $\hat\rho^\ast$ attains the minimax rate $c_{n,p}^2(\frac{\log p}{n})^{1-q}$ for any $\alpha>2\sqrt{2}$.

 Threshold estimators with small threshold constants are meaningful because of two essential reasons. On the one hand the smaller the threshold constant is the more dependency structure is captured by the estimator, on the other hand small threshold constants are of practical interest. To discuss the second reason let us suppose that we have a threshold level $10(\frac{\log p}{n})^{1/2}$,  for instance. Then, we already need a sample size larger than a hundred just to compensate the threshold constant $10$, where we have not even considered the dimension $p$ which is typically at least in the thousands. Thus, in practice moderate threshold constants may lead to a trivial estimator which provides the identity as an estimate for the correlation matrix, no matter what we actually observe.

\cite{Han2013} and \cite{Wegkamp2013} recently worked on a subject related to the issue of this article. \cite{Han2013} studied the problem of estimating the generalized latent correlation matrix of a so called transelliptical distribution, which was introduced by the same authors in \cite{Han2014}. They call a distribution transelliptical if under monotone transformations of the marginals the transformed distribution is elliptical. Then the generalized latent correlation matrix is just the correlation matrix of the transformed distribution. \cite{Wegkamp2013} investigate the related problem of estimating the copula correlation matrix in an elliptical copula model. \cite{Han2013} study the rate of convergence to the generalized latent correlation matrix for an transformed version of Kendall's tau sample correlation matrix without any additional structural assumptions on the transelliptical distribution whereas  \cite{Wegkamp2013} additionally consider copula correlation matrices with spikes. The authors propose an adaptive estimator based on Kendall's tau correlation matrix for this statistical problem. %Clearly in both models no moment assumptions are necessary.\\ 

We further investigate the elliptical model introduced above under the additional constraint that the observations are Gaussian. In this case we even establish a critical threshold constant under the squared Frobenius norm, i.e. we identify a constant $\alpha^\ast>0$ such that the proposed estimator attains the minimax rate for $\alpha>\alpha^\ast$ but in general not for $\alpha<\alpha^\ast$.  The critical value $\alpha^\ast$ for estimation of $\rho$ is given by $\frac{2\sqrt{2}}{3}$ and therefore by choosing $\alpha$ slightly larger than $\alpha^\ast$ the corresponding estimator is not only rate optimal but provides a non-trivial estimate of the true correlation matrix even for moderate sample sizes $n$. Compared to Gaussian observations, for elliptically distributed random vectors the critical threshold level is not known exactly but lies in the interval $[\frac{2\sqrt{2}}{3},2]$. Furthermore, for $\alpha<\frac{2}{3}$ the considered estimator does not even attain the rate $c_{n,p}(\frac{\log p}{n})^{1-q/2}$ over any regarded set of sparse correlation matrices. To the best of the author's knowledge this is the first work concerning critical threshold constants. 

%Simultaneously, analogous results for the estimation of Kendall's tau correlation matrix are established. Note that the estimation of Kendall's tau correlation is also of self-interest (see \cite{Embrecht2003}).

The main ingredient to provide the critical threshold level is a sharp large deviation result for Kendall's tau sample correlation if the underlying $2$-dimensional normal distribution has weak correlation between the components, which says  that even sufficiently far in the tails of the distribution of Kendall's tau sample correlation the Gaussian approximation induced by the central limit theorem for Kendall's tau sample correlation applies. This result is evolved from an asymptotic expansion of the number of permutations with a certain number of inversions (see \cite{Clark2000}).

Since Kendall's tau sample correlation is invariant under strictly increasing transformations of the components of the observations, it is easy to see that the investigation of this article covers more general statistical problems.
Specifically, the results of Section 4 extend to the estimation of sparse latent generalized correlation matrices in nonparanormal distributions (\cite{Liu2009}, \cite{Liu2012}, \cite{Xue2012}, \cite{Han2013b}). Moreover the results of Section 3 hold for the estimation of sparse latent generalized correlation matrices of meta-elliptical distributions (\cite{Fang2002}) and transelliptical distributions (\cite{Han2014}) as well as for the estimation of sparse copula correlation matrices for certain elliptical copula models (\cite{Wegkamp2013}). For ease of notation we prefer to confine the exposition to elliptical and Gaussian observations.  

%%%%%%%%%%%%%%%%%%%%%%%%%%%%%%%%%%%%%%%%%%%%%%%%%%%%%%%%%%%%%%%%%%%%%%%%%%%%%%%%%%%%%%%%%%%%%%%%%%%%%%%%%%%%%%%%%%%%%%%%%%%%%%%%%%%%%%%%%%
\subsection{Relation to other literature}
This article is related to the works of \cite{Wegkamp2013} and \cite{Han2013} about high-dimensional correlation matrix estimation based on Kendall's tau sample correlation matrix. In contrast to their works we assume the correlation matrices to be sparse and equip them with the same weak $\ell_q$-ball sparsity condition on the rows as  \cite{Caib} do for covariance matrices. We replace the sample covariance matrix in the hard threshold estimator of \cite{Bickel2008b} by a transformed version of Kendall's tau sample correlation matrix. In contrast to \cite{Caib} we are mainly interested in threshold levels for which the proposed estimator attains the minimax rate. In other words, the central question of this paper is how much information from the pilot estimate is permitted to retain under the restriction to receive a rate optimal estimator. In a manner, this enables us to recognize as much dependence structure in the data as possible without overfitting. Note that the permissible threshold constants for covariance matrix estimation in the inequalities (26) and (27) in \cite{Caib} based on \cite{Saulis1991} and \cite{Bickel2008} are not given explicitly and therefore it is even vague if any practically applicable universal threshold estimator attains the minimax rate. Besides, the calibration of the threshold level for covariance matrix estimation based on the sample covariance matrix strongly depends on the tails of the underlying distribution. In contrast, the entries of the proposed correlation matrix estimator have universal tails. 

Recently, in a semiparametric Gaussian copula model \cite{Xue2014} proved that a related threshold estimator for the correlation matrix based on Spearman's rank correlation is minimax optimal if the threshold level is equal to $40\pi(\log p/n)^{1/2}$. However, this value is far away from being of practical relevance. Hence, it is a natural question to ask how large the threshold constant in fact needs to be to get an adaptive rate optimal estimator. We answer this question in the model stated above. 

\cite{Cai2011} give a lower bound on the threshold level of a threshold estimator for covariance matrices with an entry-wise adaptive threshold level such that the minimax rate under the spectral norm is attained. However, since the entries of the proposed correlation matrix estimator in this article have universal tails, a universal threshold level seems adequate here.

As a by-product of the derivation of the minimax rate for sparse correlation matrix estimation we close a gap in the proof of the minimax lower bounds of \cite{Caib} for sparse covariance estimation since their arguments do not cover certain sparsity classes. Under Bregman divergences this affects the minimax rate for sparse covariance matrix estimation. More details are given in Section 3.

%%%%%%%%%%%%%%%%%%%%%%%%%%%%%%%%%%%%%%%%%%%%%%%%%%%%%%%%%%%%%%%%%%%%%%%%%%%%%%%%%%%%%%%%%%%%%%%%%%%%%%%%%%%%%%%%%%%%%%%%%%%%%%%%%%%%%%%%%%

\subsection{Structure of the article}
The article is structured as follows. In the next section we clarify the notation and the setting of the model. Moreover, we give a brief introduction to elliptical distributions and Kendall's tau correlation. In the third Section we discuss the minimax rates of estimation for correlation matrices under the Frobenius norm and the spectral norm in the underlying model. The fourth Section is devoted to the main results of the article. We establish the critical threshold constants for Gaussian observations regarding the minimax rates from Section 3. The proofs of the results from Section 3 and 4 are postponed to Section 7. In Section 5 we present the main new ingredients to obtain the critical threshold level of Section 4, especially we provide a sharp large deviation result for Kendall's tau sample correlation under two-dimensional normal distributions with weak correlation. Finally in Section 6 the results of the article are summarized and some related problems are discussed.

%%%%%%%%%%%%%%%%%%%%%%%%%%%%%%%%%%%%%%%%%%%%%%%%%%%%%%%%%%%%%%%%%%%%%%%%%%%%%%%%%%%%%%%%%%%%%%%%%%%%%%%%%%%%%%%%%%%%%%%%%%%%%%%%%%%%%%%%%%

%%%%%%%%%%%%%%%%%%%%%% NEW SECTION	%%%%%%%%%%%%%%%%%%%

\section{Preleminaries and model specification}\label{sec:pre}
\subsection{Notation}
We write $X\overset{d}{=}Y$ if the random variables $X$ and $Y$ have the same distribution. If $X$ is a discrete real-valued random variable then we write $\Im(X)$ for the set of all $x\in\R$ such that $\P(X=x)>0$. A sample $Y_1,...,Y_n$ of real valued random variables will be often abbreviated by $Y_{1:n}$. Moreover $\phi$ and $\Phi$ denote the probability density and cumulative distribution function of the standard normal distribution.

For a vector $x\in\R^p$ the mapping  $i\mapsto [i]_{x}=:[i]$ is a bijection on $\{1,...,p\}$ such that
\begin{align*}
|x_{[1]}|\ge...\ge |x_{[p]}|.
\end{align*}

Clearly $[\cdot]$ is uniquely determined only if $|x_i|\neq |x_j|$ for all $i\neq j$. The $q$-norm of a vector $x\in\R^p$, $q\ge 1$, is denoted by 
 \begin{align*}
 \|x\|_q:=\left(\sum_{i=1}^px_i^q\right)^{\frac{1}{q}}.
 \end{align*}
\noindent This notation will also be used if $0<q<1$. Then, of cause, $\|\cdot\|_q$ is not a norm anymore. For $q=0$ we write $\|x\|_0$ for the support of the vector $x\in \R^p$, that means $\|x\|_0$ is the number of nonzero entries of $x$, which is meanwhile a common notation from compressed sensing (see \cite{Donoho2006} and \cite{Candes2010}).

Let $f:\R\rightarrow\R$ be an arbitrary function. Then to apply the function elementwise on the matrix $A\in\R^{p\times d}$ we write  $f[A]:=(f(A_{ij}))_{i\in\{1,...,p\},j\in\{1,...,d\}}$. The Frobenius norm of $A$ is defined by $\|A\|_F^2:=\sum_{i,j}A_{ij}^2$. Moreover, the $\ell_q$ operator norm of $A$ is given by $\|A\|_{q}:=\sup_{\|x\|_q=1}\|Ax\|_q$. Recall the inequality $\|A\|_2\le \| A \|_1$ for symmetric matrices $A\in\R^{p\times p}$. For $A\in\R^{p\times p}$ we denote the $i$-th row of $A$ without the diagonal entry by $A_i\in\R^{p-1}$.

We will regularize correlation matrices by the threshold operator $T_\alpha:=T_{\alpha,n}$, $\alpha>0$, where $T_{\alpha,n}$ is defined on the set of all correlation matrices and satifisies for any correlation matrix $A\in\R^{p\times p}$ that $T_\alpha(A)_{ii}=1$ and $T_\alpha(A)_{ij}=A_{ij}\ind(|A_{ij}|>\alpha (\frac{\log p}{n})^{1/2})$, $i\neq j$.  

$C>0$ denotes a constant factor in an inequality that does not depend on any variable contained in the inequaltity, in other words for the fixed value $C>0$ the corresponding inequality holds uniformly in all variables on the left and right handside of the inequality. If we want to allow $C$ to depend on some parameter we will add this parameter to the subscript of $C$. In some computations $C$ may differ from line to line. $O,o$ are the usual Landau symbols. Finally we write $[a]$ for the largest integer not greater than $a\in \R$.

%%%%%%%%%%%%%%%%%%%%%%%%%%%%%%%%%%%%%%%%%%%%%%%%%%%%%%%%%%%%%%%%%%%%%%%%%%%%%%%%%%%%%%%%%%%%%%%%%%%%%%%%%%%%%%%%%%%%%%%%%%%%%%%%%%%%%%%%%%

\subsection{Sparsity condition}
As mentioned earlier we want to assume that the correlation matrix and Kendall's tau correlation matrix satisfy a sparsity condition. There are several possibilities to formulate such a sparsity assumption. One way is to reduce the permissible correlation matrices to a set
\begin{align*}
\G_{q}(c_{n,p}):=\{A=(a_{ij})_{1\le i,j\le p}:A=A^T,a_{jj}=1,A_i\in B_{q}(c_{n,p}),~i=1,...,p\},
\end{align*}
where $B_q(c_{n,p}),~0\le q<1$ is the $\ell_q$-ball of radius $c_{n,p}>0$. \cite{Bickel2008} used this kind of sparsity condition for covariance matrix estimation. Though, it is more handy to assume that each row $A_i$ is an element of a weak $\ell_q$-ball $B_{\text{w},q}(c_{n,p}),~0\le q<1,~c_{n,p}> 0$ instead, in other words for $x:=A_i$ we have $|x_{[i]}|^q\le c_{n,p}i^{-1}$.
Weak $\ell_q$-balls were originally introduced by \cite{Abramovich2006} in a different context. Note that the $\ell_q$-ball $B_q(c_{n,p})$ is contained in the weak $\ell_q$-ball $B_{\text{w},q}(c_{n,p})$. Nevertheless the complexity of estimating a correlation matrix over $\G_{\text{w},q}(c_{n,p})$ is the same as over $\G_{q}(c_{n,p})$,
where
\begin{align*}
\G_{\text{w},q}(c_{n,p}):=\{A=(a_{ij})_{1\le i,j\le p}:A=A^T,a_{jj}=1,A_i\in B_{\text{w},q}(c_{n,p}),~i=1,...,p\}.
\end{align*}

The reader is referred to \cite{Caib} for analogous statements for covariance matrix estimation. Therefore throughout the paper we will consider the case that $\rho\in \G_{\text{w},q}(c_{n,p}).$

%%%%%%%%%%%%%%%%%%%%%%%%%%%%%%%%%%%%%%%%%%%%%%%%%%%%%%%%%%%%%%%%%%%%%%%%%%%%%%%%%%%%%%%%%%%%%%%%%%%%%%%%%%%%%%%%%%%%%%%%%%%%%%%%%%%%%%%%%%

\subsection{Elliptical distributions}
Commonly a random vector $Y\in \R^p$ is called elliptically distributed if its characteristic function $\varphi_Y$ is  of the form
\begin{align*}
\varphi_Y(t)=e^{it^T\mu}\psi(t^T \bar\Sigma t),~t\in\R^p,
\end{align*}
for a positive semi-definite matrix $\bar\Sigma\in\R^{p\times p}$ and a function $\psi$.
The following representation for elliptically distributed random vectors will be convenient for our purposes, which may be found for example in \cite{Fang1990}:
\begin{proposition}\label{pro:ell}
A random vector $X\in\R^p$ has an elliptical distribution iff for a matrix $A\in\R^{p\times q}$ with $\rank(A)=q$, a vector $\mu\in\R^p$, a non-negative random variable $\xi$ and random vector $U\in \R^q$, where $\xi$ and  $U$ are independent and $U$ is uniformly distributed on the unit sphere $\S^{q-1}$, $X$ has the same distribution as $\mu+\xi A U$.
\end{proposition}

Therefore we write $X=(X_1,...,X_p)^T\sim EC_p(\mu,\Sigma,\xi)$ if $X$ has the same distribution as $\mu+\xi A U$, where $A\in\R^{p\times q}$ satisfies $AA^T=\Sigma$.

%%%%%%%%%%%%%%%%%%%%%%%%%%%%%%%%%%%%%%%%%%%%%%%%%%%%%%%%%%%%%%%%%%%%%%%%%%%%%%%%%%%%%%%%%%%%%%%%%%%%%%%%%%%%%%%%%%%%%%%%%%%%%%%%%%%%%%%%%%

\subsection{Kendall's tau}
Let $(Y,Z)$ be a two-dimensional random vector. We denote Kendall's tau correlation between $Y$ and $Z$ by
\begin{align*}
\tau(Y,Z):&=\P((Y-\tilde Y)(Z-\tilde Z)>0)-\P((Y-\tilde Y)(Z-\tilde Z)<0)\\
&=\E \sign(Y-\tilde Y)(Z-\tilde Z)\\
&=\Cov(\sign(Y-\tilde Y),\sign(Z-\tilde Z)),
\end{align*}
where $(\tilde Y,\tilde Z)$ is an independent copy of $(Y,Z)$ and $\sign x:=\ind(x>0)-\ind(x<0)$. Its empirical version based on an i.i.d. sample $(Y_1,Z_1),...,(Y_n,Z_n)\overset{d}{=}(Y,Z)$ is called Kendall's tau sample correlation and given by
\begin{align*}
\hat\tau(Y_{1:n},Z_{1:n}):=\frac{1}{n(n-1)}\sum_{\substack{k,l=1\\k\neq l}}^n\sign(Y_k-Y_l)\sign(Z_k-Z_l).
\end{align*}
Analogously we write $\rho(X,Y)$ for the (Pearson's) correlation between $X$ and $Y$. For a nondegenerate elliptically distributed random vector $(X,Y)\sim EC_p(\mu,\Sigma,\xi)$ we define the (generalized) correlation $\rho(X,Y)$ even if the second moments of the components do not exist. Thereto let $\rho(X,Y):=\frac{\Sigma_{12}}{\sqrt{\Sigma_{11}\Sigma_{22}}}$. Obviously, this expression is equivalent to the usual definition of correlation if the second moments exist.\\
Let $X_1=(X_{11},...,X_{p1})^T$ be a $p$-dimensional random vector. We denote the Kendall's tau correlation matrix of $X_1$ by $\tau:=(\tau_{ij})$, where $\tau_{ij}=\tau(X_{i1},X_{j1})$. Thus, $\tau$ is positive semidefinite since
\begin{align*}
\tau=\Cov(\sign[X_1-\tilde X_1],\sign[X_1- \tilde X_1]).
\end{align*}
Moreover for an i.i.d. sample $X_1,...,X_n\in\R^p$ we call $\hat\tau=(\hat\tau_{ij})$ with 
\begin{align*}
\hat\tau:=\frac{1}{n(n-1)}\sum_{k.l=1}^n\sign[X_k-X_l]\sign[X_k-X_l]^T
\end{align*}
Kendall's tau sample correlation matrix. Hence $\hat\tau$ is positive semi-definite. Furthermore if the distributions of the components of $X_1$ have no atoms, then $\tau$ (resp. $\hat\tau$) is (a.s.) a correlation matrix. Note that the distributions of the components of $X_1$ have no atoms iff either $\rank(\Sigma)=1$ and the distribution of $\xi$ has no atoms or $\rank(\Sigma)\ge2$ and $\P(\xi=0)=0$. If $EC_p(\mu,\Sigma,\xi)$ is a distribution with non-atomic components then Kendall's tau correlation matrix $\tau$ is determined by the correlation matrix $\rho$, particularly we have $\rho=\sin\left[\frac{\pi}{2}\tau\right]$ (see \cite{Hult2002}). Hence, $\sin\left[\frac{\pi}{2}\hat\tau\right]$ is a natural estimator for $\rho$. In the following we will occasionally use the next two elementary lemmas to connect the correlation matrix to Kendall's tau correlation matrix:
\begin{lemma}\label{lem:boukencor}
For any matrices $A,B\in\R^{p\times p}$ holds
 \begin{align*}
 \|\sin[\frac{\pi}{2}A]-\sin[\frac{\pi}{2}B]\|_{F}&\le \frac{\pi}{2}\|A-B\|_{F},\\
  \|\sin[\frac{\pi}{2}A]-\sin[\frac{\pi}{2}B]\|_{1}&\le \frac{\pi}{2}\|A-B\|_{1}.
 \end{align*}
 \end{lemma}
\begin{proof}
The function $x\mapsto \sin(\frac{\pi}{2}x)$ is Lipschitz continuous with Lipschitz constant $L=\frac{\pi}{2}$. Therefore we conclude
\begin{align*}
 \|\sin[\frac{\pi}{2}A]-\sin[\frac{\pi}{2}B]\|_{F}^2&=\sum_{i,j}(\sin(\frac{\pi}{2}A_{ij})-\sin(\frac{\pi}{2}B_{ij}))^2\\
 &\le \sum_{i,j}\frac{\pi^2}{4}(A_{ij}-B_{ij})^2=\frac{\pi^2}{4}\|A-B\|_{F}^2.
\end{align*}
Analogously,
\begin{align*}
 \|\sin[\frac{\pi}{2}A]-\sin[\frac{\pi}{2}B]\|_{1}&=\max_{j=1,\dots,p}\sum_{i=1}^p|\sin(\frac{\pi}{2}A_{ij})-\sin(\frac{\pi}{2}B_{ij})|\\
 &\le \max_{j=1,\dots,p} \frac{\pi}{2}\sum_{i=1}^p|A_{ij}-B_{ij}|\\
 &=\frac{\pi}{2}\|A-B\|_1.
\end{align*}
\end{proof}
\begin{lemma}\label{lem:spatra}
Let $\tau\in \G_{\text{w},q}(c_{n,p})$ for $c_{n,p}>0$, then $\sin[\frac{\pi}{2}\tau]\in \G_{\text{w},q}\left((\frac{\pi}{2})^q c_{n,p}\right)$. On the other hand, for $\sin[\frac{\pi}{2}\tau]\in \G_{\text{w},q}(c_{n,p})$ holds $\tau\in \G_{\text{w},q}(c_{n,p})$.
\end{lemma}
\begin{proof}
The first statement follows easily by the fact that the derivative of the sine function is bounded by 1. The second statement is obtained by concavity of the sine function on $\left[0,\frac{\pi}{2}\right]$ and convexity of the sine function on $\left[-\frac{\pi}{2},0\right]$.
\end{proof}

%%%%%%%%%%%%%%%%%%%%%%%%%%%%%%%%%%%%%%%%%%%%%%%%%%%%%%%%%%%%%%%%%%%%%%%%%%%%%%%%%%%%%%%%%%%%%%%%%%%%%%%%%%%%%%%%%%%%%%%%%%%%%%%%%%%%%%%%%%  

\subsection{Further model specification and the regarded threshold estimators}
   For a better overview we summarize the assumptions of the results in Section $\ref{sec:min}$ and $\ref{sec:cri}$. 

\begin{itemize}
\item[$(A_1)$] $X_1,...,X_n\widesim{i.i.d.} EC_p(\mu,\Sigma,\xi)$ such that the distributions of the components $X_{i1}$ have no atoms.
\item[$(A_1^\ast)$] $X_1,...,X_n\widesim{i.i.d.} \mathcal{N}_p(\mu,\Sigma)$ such that $\Sigma_{ii}>0$ for all $i=1,...,p$.
\item[$(A_2)$] The parameters of the set  $\mathcal{G}_{\text{w},q}(c_{n,p})$ satisfy $0\le q <1$ and $c_{n,p}>0$. For ease of notion $c_{n,p}$ is a postive integer if $q=0$. 
\item[$(A_3)$] There exists a constant $M>0$ such that
\[c_{n,p} \le  M n^{(1-q)/2}(\log p)^{-(3-q)/2}.\]
\item[$(A_3^\ast)$] There exists a constant $v>0$ such that
\[c_{n,p}\le \left(v\left(\frac{\log p}{n}\right)^{q/2}\right)\wedge \left(M n^{(1-q)/2}(\log p)^{-(3-q)/2}\right).\]
\item[$(A_4)$] There exists a constant $m<0$ such that the radius $c_{n,p}\ge m\left(\frac{\log p}{n}\right)^{q/2}$.
\item[$(A_5)$] There exists a constant $\eta_l>1$ such that $p>n^{\eta_l}>1$.
\item[$(A_6)$] There exists a constant $\eta_u>1$ such that $p<n^{\eta_u}$.
\end{itemize}

The Assumptions $(A_1), (A_3)$ and $(A_5)$ are sufficient to ensure that the minimax lower bound is true. For an upper bound on the maximal risk of the considered estimators $(A_3)$ and $(A_5)$ are not required. Assumptions $(A_1^\ast)$ and $(A_6)$ guarantee that the entries $\hat\tau_{ij}$ based on components with weak correlation satisfy a Gaussian approximation even sufficiently far in the tails. This is essential to provide the critical threshold level.

Our investigation involves two highly related threshold estimators based on Kendall's tau sample correlation matrix $\hat\tau$. %For the estimation of Kendall's tau correlation matrix we investigate $\hat\tau^\ast:=\hat\tau^\ast_\alpha:=T_\alpha(\hat\tau)$. 
Based on Kendall's tau correlation matrix two natural estimators appear for correlation matrix estimation. We consider both, $\hat\rho^\ast:=\hat\rho^\ast_\alpha:=\sin[\frac{\pi}{2}\hat\tau^\ast]$ with $\hat\tau^\ast:=\hat\tau^\ast_\alpha:=T_\alpha(\hat\tau)$ from the introductory section and $\hat\rho:=\hat\rho_\alpha:=T_\alpha(\sin[\frac{\pi}{2}\hat\tau])$. The difference between them is the order of thresholding and transformation by the sine function. Technically it is favorable to deduce the properties of $\hat\rho$ from $\hat\rho^\ast$.

In the next section $\hat\rho$ may also denote an arbitrary estimator. The meaning of the notation is obtained from the context.

%%%%%%%%%%%%%%%%%%%%%% NEW SECTION	%%%%%%%%%%%%%%%%%%%

\section{Minimax rates of estimation for sparse correlation matrices}\label{sec:min}

As already mentioned before we want to study the minimax rates under the squared Frobenius norm and the squared spectral norm such that $\rho$  lies in a fixed class $\G_{\text{w},q}(c_{n,p})$, i.e. we bound
\begin{align*}
\underset{\hat\rho}{\inf} \underset{\rho\in \G_{\text{w},q}(c_{n,p})}{\sup}\frac{1}{p}\E\|\hat\rho-\rho\|_F^2~\text{ and }~\underset{\hat\rho}{\inf} \underset{\rho\in \G_{\text{w},q}(c_{n,p})}{\sup}\E\|\hat\rho-\rho\|_2^2,
\end{align*} 
where the infimum is taken over all estimators for $\rho$. In the supremum we have a slight abuse of notation, since the maximal risks
\begin{align*}
 \underset{\rho\in \G_{\text{w},q}(c_{n,p})}{\sup} \frac{1}{p}\E\|\hat\rho-\rho\|_F^2~\text{ and }~\underset{\rho\in \G_{\text{w},q}(c_{n,p})}{\sup}\E\|\hat\rho-\rho\|_2^2
 \end{align*}
  of a fixed estimator $\hat\rho$ are to read as the supremum over all permissible elliptical distributions $EC_p(\mu,\Sigma,\xi)$ for the underlying sample $X_1,...,X_n$ of i.i.d. observations such that $\rho$ lies in  $\G_{\text{w},q}(c_{n,p})$. %Notice that $\rho$ and $\tau$ do only depend on $\Sigma$.
	
We first present the sharp minimax lower bounds of estimation for the correlation matrix.
\begin{theorem}[Minimax lower bound of estimation for sparse correlation matrices]\label{the:minlowboucor}\label{the:minlowbouken}
Under the assumptions $(A_1)-(A_3)$ and $(A_5)$ the following minimax lower bounds hold\\[-0.45cm]
\begin{align}\label{eqn:minlowsam}
\underset{\hat\rho}{\inf} \underset{\rho\in \G_{\text{w},q}(c_{n,p})}{\sup}\frac{1}{p}\E\|\hat\rho-\rho\|_F^2\ge C_{M,\eta_l,q}c_{n,p}\left(c_{n,p}^{\frac{2}{q}}\wedge\frac{\log p}{n}\right)^{1-\frac{q}{2}}
\end{align}
 \\[-0.4cm]
\noindent and
 \\[-0.5cm]
\begin{align}\label{eqn:minlowsam2}
\underset{\hat \rho}{\inf} \underset{\rho\in \G_{\text{w},q}(c_{n,p})}{\sup}\E\|\hat\rho-\rho\|_2^2\ge \tilde C_{M,\eta_l,q}c_{n,p}^2\left(c_{n,p}^\frac{2}{q}\wedge\frac{\log p}{n}\right)^{1-q}
\end{align}
for some constants $C_{M,\eta_l,q},\tilde C_{M,\eta_l,q}>0$.
\end{theorem}
The lower bound for estimating $\rho$ over some class $\G_{\text{w},q}(c_{n,p})$ with radius $c_{n,p}>2v^{q}\left(\frac{\log p}{n}\right)^{q/2}$, $v>0$ sufficiently small, is an immediate consequence of the proof of Theorem 4 in \cite{Caib}, where the authors use a novel generalization of Le Cam's method and Assouad's lemma to treat the two-directional problem of estimating sparse covariance matrices. In the proof the authors consider the minimax lower bound over a finite subset of sparse covariance matrices $\mathcal{F}^\ast$ for a normally distributed sample, where the diagonal entries are equal to one. Therefore their proof holds for estimation of correlation matrices too. \cite{Xue2014} use the same argument to transfer the minimax lower bound of \cite{Caib} under the $\ell_1$ and $\ell_2$ operator norms to correlation matrix estimation. But, they do not take into account that the result does not hold for all classes $\G_{\text{w},q}(c_{n,p})$. If $2v^{q}\left(\frac{\log p}{n}\right)^{q/2}\ge c_{n,p}$ the value $k$ on page 2399 of \cite{Caib} is equal to zero which leads to a trivial lower bound. Hence, in this case one needs a different construction of the finite subset of sparse correlation matrices over which the expected loss of an estimator $\hat\rho$ is maximized.
\begin{remark}%~\\[-0.5cm]
%\begin{itemize}
%\item[1.] 
While the mentioned gap in the arguments of \cite{Caib} has no effect on the minimax lower bound for covariance matrix estimation under operator norms $\ell_w$,~$w=1,2$, under Bregman divergences the minimax rate is given by
\begin{align}
c_{n,p}\left(c_{n,p}^\frac{2}{q}\wedge\frac{\log p}{n}\right)^{1-\frac{q}{2}}+\frac{1}{n}.\end{align}
Note that the minimax rate for $q=0$ remains unaffected by this consideration.
%\item[2.] By slight adjustments of the arguments one can prove the following (optimal) minimax lower bound for correlation matrix estimation under the $\ell_w,~w=1,2,$ %operator norm
%\begin{align}
%\underset{\hat\rho}{\inf} \underset{\rho\in \G_{\text{w},q}(c_{n,p})}{\sup}\E\|\hat\rho-\rho\|_{w}^2\ge C_{M,\eta_l,q}c_{n,p}^2\left(c_{n,p}^{\frac{2}{q}}\wedge\frac{\log p}{n}\right)^{1-%q}.
%\end{align}
%\end{itemize}
\end{remark}

The following proposition shows that estimating the correlation matrix by the identity matrix achieves the minimax rate if $c_{n,p}=O\big(\frac{\log p}{n}\big)^{q/2}$. So, if correlation matrices with entries of order $(\frac{\log p}{n})^{1/2}$ are excluded from the sparsity class $\G_{\text{w},q}(c_{n,p})$, in a minimax sense there are no estimators which perform better than the trivial estimate by the identity matrix. 
\begin{proposition}\label{pro:ult}
For $q>0$ and absolute constants $\tilde C,C>0$ holds
\begin{align}\label{eqn:ult2}
%\underset{\tau\in \G_{\text{w},q}(c_{n,p})}{\sup}\frac{1}{p}\E\|\Id-\tau\|_F^2\le C c_{n,p}^{\frac{2}{q}}
%\hspace{0.5cm}\text{and}\hspace{0.5cm}
\underset{\rho\in \G_{\text{w},q}(c_{n,p})}{\sup}\frac{1}{p}\|\Id-\rho\|_F^2&\le C c_{n,p}^{\frac{2}{q}},\\
\underset{\rho\in \G_{\text{w},q}(c_{n,p})}{\sup}\|\Id-\rho\|_2^2&\le \tilde C c_{n,p}^{\frac{2}{q}}.\label{eqn:ult3}
\end{align}
Especially, under the assumptions $(A_1),(A_2)$ and $(A_3^\ast)$ the identity matrix attains the minimax rate for estimating $\rho$. 
\end{proposition}

Henceforth, we only consider classes $\mathcal{G}_{\text{w},q}(c_{n,p})$ of sparse correlation matrices, such that $c_{n,p}\ge m(\frac{\log p}{n})^{q/2}$ for some positive constant $m>0$. In this case the estimation of the correlation matrix can be accomplished by the threshold estimator   $\hat\rho^\ast$ for suitable threshold levels $\alpha(\frac{\log p}{n})^{1/2}$. The following theorem proves that the choice $\alpha>2$ is sufficient to guarantee that the estimator achieves the minimax rate under the squared Frobenius norm, $\alpha>2\sqrt{2}$ is sufficient to achieve the minimax rate under the squared spectral norm. Hence, the proposed estimator provides even for small sample sizes a non-trivial estimate of the correlation matrix, where by ``non-trivial estimate'' we understand that with positive probability the estimator is not the identity matrix.
\begin{theorem}[Minimax upper bound for Kendall's tau correlation matrix estimation]\label{the:minuppbouken}\label{the:minuppboucor}
Under the assumptions $(A_1)-(A_5)$ the threshold estimator $\hat\rho^\ast=\sin[\frac{\pi}{2}\hat\tau^\ast],~\hat\tau^\ast=T_\alpha(\hat\tau),$ based on Kendall's tau sample correlation matrix $\hat\tau$ attains the minimax rate over the set $\G_{\text{w},q}(c_{n,p})$ for a sufficiently large threshold constant $\alpha$, particularly
\begin{align}\label{eqn:minuppbouken}
\underset{\rho\in \G_{\text{w},q}(c_{n,p})}{\sup}\frac{1}{p}\E\|\hat\rho^\ast-\rho\|_F^2\le C_{\alpha}c_{n,p}\left(\frac{\log p}{n}\right)^{1-\frac{q}{2}}~~\text{if }\alpha>2,
\end{align}
 \\[-0.4cm]
\noindent and
 \\[-0.5cm]
\begin{align}\label{eqn:minuppboucor}
\underset{\rho\in \G_{\text{w},q}(c_{n,p})}{\sup}\frac{1}{p}\E\|\hat\rho^\ast-\rho\|_2^2\le \tilde C_{\alpha}c_{n,p}^2\left(\frac{\log p}{n}\right)^{1-q}~~\text{if }\alpha>2\sqrt{2},
\end{align}
where the constants $\tilde C_\alpha, C_{\alpha}>0$ depend on the threshold constant $\alpha$ only.
\end{theorem}

Altogether, Proposition \ref{pro:ult} and Theorem \ref{the:minuppbouken} verify that Theorem \ref{the:minlowbouken} provides the minimax rates of estimation for sparse correlation matrices.

As we will see in the next section, inequality \eqref{eqn:minuppbouken} cannot be extended to any threshold constant $\alpha<2$ without using an improved large deviation inequality in comparison to Hoeffding's inequality for U-statistics (\cite{Hoeffding1963}). Before we start to discuss the issue of critical threshold constants, we close the section with a final result. We show that it is irrelevant whether we first apply a threshold operator on Kendall's tau sample correlation matrix and afterwards transform the obtained matrix to an appropriate estimator for the correlation matrix or vice versa where in the latter procedure the threshold constant needs to be adjusted to $\frac{\pi}{2}\alpha$. Heuristically, this seems evident, as asymptotically the ratio between the implied threshold level of $\sin[\frac{\pi}{2}T_{\alpha}(\hat\tau)]$ with respect to entries of $\sin[\frac{\pi}{2}\hat\tau]$ and the threshold level of $T_{\frac{\pi}{2}\alpha}(\sin[\frac{\pi}{2}\hat\tau])$ is one. Nevertheless, $\hat\rho^\ast$ is already for smaller sample sizes a non-trivial estimator for the correlation matrix than $\hat\rho$.
\begin{theorem}\label{the:minuppboucor2}
Under the assumptions $(A_1)-(A_5)$ the threshold estimator $\hat\rho:=T_\alpha(\sin[\frac{\pi}{2}\hat\tau])$ based on Kendall's tau sample correlation matrix $\hat\tau$ attains the minimax rate over the set $\G_{\text{w},q}(c_{n,p})$ for a sufficiently large threshold constant $\alpha$, particularly
\begin{align}\label{eqn:minuppboucor2}
\underset{\rho\in \G_{\text{w},q}(c_{n,p})}{\sup}\frac{1}{p}\E\|\hat\rho-\rho\|_F^2&\le C_{\alpha}c_{n,p}\left(\frac{\log p}{n}\right)^{1-\frac{q}{2}}~~\text{if }\alpha>\pi,%\\
%\underset{\rho\in \G_{\text{w},q}(c_{n,p})}{\sup}\E\|\hat\rho-\rho\|_2^2&\le C_{\alpha}c_{n,p}^2\left(\frac{\log p}{n}\right)^{1-q}~~\text{if }\alpha>\frac{3}{2}\pi,\label{eqn:minuppboucorr3}
\end{align}
where the factor $C_{\alpha}>0$ depends on the threshold constant $\alpha$.
\end{theorem}

%%%%%%%%%%%%%%%%%%%%%% NEW SECTION	%%%%%%%%%%%%%%%%%%%

\section{Critical threshold levels for minimax estimation}\label{sec:cri}
In this section we discuss the critical threshold levels for which the threshold estimators just attain the minimax rate. In opposite to the previous section we need throughout that the observations are Gaussian random vectors and the dimension $p$ does not grow to fast in $n$. Precisely, $p$ is only allowed to grow polynomially in $n$. Otherwise we cannot apply the large deviation inequality $(\ref{eqn:conweacor})$ from section 5. As mentioned earlier the arguments in the proof of Theorem \ref{the:minuppboucor} are optimized up to the concentration inequality for the entries of Kendall's tau sample correlation matrix. This is supposed to mean that if inequality $(\ref{eqn:conweacor})$ would hold for any entry $\hat\tau_{ij}$ then replacing Hoeffding's inequality by $(\ref{eqn:conweacor})$ at some places of the proof already provides the critical threshold level. Actually, we do not need inequality $(\ref{eqn:conweacor})$ for all entries $\hat\tau_{ij}$. It is sufficient to stay with Hoeffding's inequality for strongly correlated components. We say that two real-valued random variables $(Y,Z)$ are strongly correlated (with respect to $p$ and $n$) if $\rho(Y,Z)\ge\frac{5\pi}{2}(\frac{\log p}{n})^{1/2}$. Otherwise $(Y,Z)$ are called weakly correlated. The constant $\frac{5\pi}{2}$ is chosen arbitrarily and could be replaced by any sufficiently large value. Our choice guarantees that for any threshold level $\alpha(\frac{\log p}{n})^{1/2}$ with $\frac{1}{2^{1/2}}\alpha>\frac{2}{3}$ the entries $\hat\tau_{ij}$ corresponding to strongly strongly correlated components are not rejected by the threshold operator with probability $1-Cp^{-\gamma}$, where $\gamma>0$ is suitably large. Within the weakly correlated components we have a transition to entries $\hat\tau_{ij}$, which are very likely to be rejected. Therefore it is sufficient to use only in the latter case more precise large deviation results instead of Hoeffding's inequality.

For the proof that the threshold estimator $\hat\rho^\ast$ with threshold constant $\alpha<\frac{2\sqrt{2}}{3}$ does not achieve the minimax rate over some classes $\G_{\text{w},q}(c_{n,p})$ we just have to study $ p^{-1}\E\|\hat\rho^\ast-\Id\|_F^2$ by the lower bound in inequality $(\ref{eqn:conweacor})$, where the identity matrix is the underlying correlation matrix. Note that obviously $\Id\in\G_{\text{w},q}(c_{n,p})$ for any sparsity class $\G_{\text{w},q}(c_{n,p})$.

Similar to the previous section we first present the results for the estimator $\hat\rho^\ast$. Afterwards analogous statements for $\hat\rho$ are formulated.

\begin{theorem}\label{the:optuppbou}
Under the assumptions $(A_1^\ast)$ and $(A_2)-(A_6)$ let $\alpha\neq \frac{2\sqrt{2}}{3}$, then $\hat\rho^\ast$ is minimax rate optimal under the squared Frobenius norm over all sets $\G_{\text{w},q}(c_{n,p})$ iff $\alpha>\frac{2\sqrt{2}}{3}$. Hence, for $\alpha>\frac{2\sqrt{2}}{3}$ and an arbitrary set  $\G_{\text{w},q}(c_{n,p})$ we have
%\begin{align}\label{eqn:optuppbouken}
%\underset{\tau\in \G_{\text{w},q}(c_{n,p})}{\sup}\frac{1}{p}\E\|\hat\tau^\ast-\tau\|_F^2\le C_{\alpha,\eta_u}c_{n,p}\left(\frac{\log p}{n}\right)^{1-\frac{q}{2}}
%\end{align}
 %\\[-0.4cm]
%\noindent and
 %\\[-0.5cm]
\begin{align}\label{eqn:optuppboucor}
\underset{\rho\in \G_{\text{w},q}(c_{n,p})}{\sup}\frac{1}{p}\E\|\hat\rho^\ast-\rho\|_F^2\le  C_{\alpha,\eta_u}c_{n,p}\left(\frac{\log p}{n}\right)^{1-\frac{q}{2}},
\end{align}
where the constants $C_{\alpha,\eta_u}>0$ depends on $\alpha$ and $\eta_u$.

Moreover, for $\alpha<\frac{2}{3}$ there is no permissible set $\G_{\text{w},q}(c_{n,p})$ such that $\hat\rho^\ast$ attains the minimax rate over $\G_{\text{w},q}(c_{n,p})$.
\end{theorem}
\begin{theorem}\label{the:optuppbou2}
Under the assumptions $(A_1^\ast)$ and $(A_2)-(A_6)$ let $\alpha\neq \frac{\sqrt{2}\pi}{3}$, then $\hat\rho$ is minimax rate optimal over all sets $\G_{\text{w},q}(c_{n,p})$ iff $\alpha>\frac{\sqrt{2}\pi}{3}$. Hence, for $\alpha>\frac{\sqrt{2}\pi}{3}$ and an arbitrary set  $\G_{\text{w},q}(c_{n,p})$ we have
\begin{align}\label{eqn:optuppbou2}
\underset{\rho\in \G_{\text{w},q}(c_{n,p})}{\sup}\frac{1}{p}\E\|\hat\rho-\rho\|_F^2\le C_{\alpha,\eta_u}c_{n,p}\left(\frac{\log p}{n}\right)^{1-\frac{q}{2}},
\end{align}
where the constant $C_{\alpha,\eta_u}>0$ depends on $\alpha$ and $\eta_u$.

Moreover, for $\alpha<\frac{\pi}{3}$ there is no permissible set $\G_{\text{w},q}(c_{n,p})$ such that $\hat\rho$ attains the minimax rate over $\G_{\text{w},q}(c_{n,p})$.
\end{theorem}
So far we are only able to establish the critical threshold constant for $\hat\rho^\ast$ under the squared Frobenius norm. Under the squared spectral norm the  following slightly weaker result holds.
\begin{theorem}\label{optspe}
Let the assumptions $(A_1^\ast)$ and $(A_2)-(A_6)$ hold. Then, $\hat\rho^\ast$ is minimax rate optimal under the squared spectral norm over all sets $\G_{\text{w},q}(c_{n,p})$ if $\alpha>4/3$.  For $\alpha<\frac{2\sqrt{2}}{3}$ there exists a permissible class  $\G_{\text{w},q}(c_{n,p})$ such that $\hat\rho^\ast$ does not attain the minimax rate.
\end{theorem}

%%%%%%%%%%%%%%%%%%%%%% NEW SECTION	%%%%%%%%%%%%%%%%%%%

\section{On Kendall's tau sample correlation for normal distributions with weak correlation}\label{sec:ken}
In this section we discuss the properties of the tails of Kendall's tau sample correlation based on a sample $(Y_1,Z_1),...,(Y_n,Z_n)\sim \mathcal{N}_2(\mu,\Sigma)$ for $\Sigma=\Id$ and small perturbations of the identity. Specifically, we need preferably sharp upper and lower bounds on its tails.
The essential argument for our investigation is the natural linkage between Kendall's tau sample correlation and the number of inversions in a random permutation. So we can apply ``an asymptotic expansion for the number of permutations with a certain number of inversions'' developed by \cite{Clark2000}.

%%%%%%%%%%%%%%%%%%%%%%%%%%%%%%%%%%%%%%%%%%%%%%%%%%%%%%%%%%%%%%%%%%%%%%%%%%%%%%%%%%%%%%%%%%%%%%%%%%%%%%%%%%%%%%%%%%%%%%%%%%%%%%%%%%%%%%%%%%

\subsection{Kendall's tau sample correlation for the standard normal distribution}
Before studying the tails of Kendall's tau sample correlation $\hat\tau(Y_{1:n},Z_{1:n})$ based on a sample $(Y_1,Z_1),...,(Y_n,Z_n)\sim \mathcal{N}_2(\mu,\Id)$, we first prove that  $\hat\tau(Y_{1:n},Z_{1:n})$ has - after centering and rescaling -  the same distribution as the the number of inversions in a random permutation on $\{1,...,n\}$. This result is probably known for a long time but to the best of the author's knowledge in no work mentioned explicitly. Particularly, so far statisticians have not taken advantage of any developments on inversions in random permutations in this context.

The number of inversions in a permutation is an old and well-studied object. We say that a permutation $\pi$ on $\{1,...,n\}$ has an inversion at $(i,j)$, $1\le i<j\le n$, iff $\pi(j)>\pi(i)$. This concept was originally introduced by \cite{Cramer1750} in the context of the Leibniz formula for the determinant of quadratic matrices. Denote by $I_n(k)$ the number of permutations on $\{1,...,n\}$ with exactly $k$ inversions.  The generating function $G$ for the numbers $I_n(k)$ is given by $G(z)=\prod_{l=1}^{n-1}\left(1+z+...+z^l\right)$ as already known at least since \cite{Muir1900}. Note that \cite{Kendall1938} studied the generating function of $\hat\tau(Y_{1:n},Z_{1:n})$ independently on prior works on inversions in permutations and thereby derived a central limit theorem for $\hat\tau(Y_{1:n},Z_{1:n})$ when $n$ tends to infinity. However such a result is not strong enough for our purposes. We actually need a Gaussian approximation for the tails of order $(\frac{\log p}{n})^{1/2}$. This will be concluded by the work of \cite{Clark2000}, who gives an asymptotic expansion for $I_n(k)$, where $k=\frac{n(n-1)}{4}\pm l$ and $l$ is allowed to grow moderately with $n$. Therefore, at this point we need that $p$ is not increasing faster than polynomially in $n$. In other words, $\frac{\log p}{\log n}$ is bounded above by some absolute constant.

Certainly, one could show the connection between Kendall's tau correlation and the number of inversion in random permutations by their generating functions. We prefer a direct proof which is more intuitive.
\begin{proposition}\label{pro:keninv}
Let $I_-$ be the number of inversions in a random permutation on $\{1,...,n\}$, $n\ge2$, and $\hat\tau(Y_{1:n},Z_{1:n})$ be Kendall's sample correlation based on $(Y_1,Z_1),...,(Y_n,Z_n)\widesim{i.i.d.}\mathcal{N}_2(\mu,\Id)$. Then, it holds 
\begin{align*}
\hat\tau(Y_{1:n},Z_{1:n})\overset{\mathcal D}{=} 1-\frac{4}{n(n-1)}I_-.
\end{align*}
\end{proposition}
\begin{proof}
Recall that by definition
\begin{align*}
\hat\tau(Y_{1:n},Z_{1:n})
&=\frac{1}{n(n-1)}\sum_{i,j}\sign(Y_i-Y_j)\sign(Z_i-Z_j).
\end{align*}

Now let $\pi:i\mapsto [i]_{Z_{1:n}}=:[i]$ be the permutation induced by the order statistics $Z_{[1]}\ge...\ge Z_{[n]}$. Therefore, rewrite
\begin{align*}
\hat\tau(Y_{1:n},Z_{1:n})&=\frac{1}{n(n-1)}\sum_{i,j}\sign(Y_{[i]}-Y_{[j]})\sign(Z_{[i]}-Z_{[j]})\\
&=\frac{2}{n(n-1)}\sum_{[i]>[j]}\sign(Y_{[i]}-Y_{[j]})\sign(Z_{[i]}-Z_{[j]})\\
&=\frac{2}{n(n-1)}\sum_{[i]>[j]}\sign(Y_{[i]}-Y_{[j]})\hspace{1.32cm}\text{(a.s.)}\\
&\overset{\mathcal D}{=}\frac{2}{n(n-1)}\sum_{i>j}\sign(Y_i-Y_j)\hspace{0.63cm}\text{(by independence of }Y\text{ and }Z\text{)}
\end{align*}

Remap $\pi:i\mapsto [i]_{Y_{1:n}}=:[i]$ by the permutation induced by the order statistics $Y_{[1]}\ge...\ge Y_{[n]}$. Obviously, $\sign(X_i-X_j)$ is $-1$ if $\pi$ has an inversion at $(i,j)$. Otherwise $\sign(X_i-X_j)$ is $1$. Denote by  $I_{-}$ the number of inversions in $\pi$ and by  $I_{+}$ the number of all the other pairs $(i,j)$. Clearly, $I_{-}+I_{+}=\binom{n}{2}$. 

Finally, we conclude
\begin{align*}
\hat\tau(Y_{1:n},Z_{1:n})&\overset{\mathcal D}{=}\frac{2}{n(n-1)}\sum_{i>j}\sign(Y_i-Y_j)\\
&=\frac{2}{n(n-1)}\left(I_{+}-I_{-}\right)=1-\frac{4}{n(n-1)}I_{-}.
\end{align*}
\end{proof}

So far we have only used that the components of the standard normal distribution are non-atomic and independent. In the next subsection we will exploit further properties of the normal distribution which makes it difficult to extend the results to further distributions.

Now we reformulate the result of \cite{Clark2000} for the number of permutations whose number of inversions differ exactly by $l$ from $\frac{n(n-1)}{4}$.
\begin{theorem}[\cite{Clark2000}]\label{the:cla}
Fix $\lambda>0$. Let $m=[\lambda^2/2]+2$ and $l\in \Im(I_{-}-\E I_-)$,
then we have
\begin{align}\label{eqn:cla}
\P(|I_{-}-\E I_-|=l)&=12(2\pi)^{-1/2}n^{-3/2}e^{-18l^2/n^3}+r_{n,\lambda,1}(l)+r_{n,\lambda,2}(l),
\end{align}
where the error terms $r_{n,\gamma,1}$ and $r_{n,\gamma,2}$ satisfy for a certain constant $C_\lambda>0$
\begin{align*}
|r_{n,\lambda,1}(l)|&\le C_\lambda\left(n^{-5/2}e^{-18l^2/n^3}n^{-6m+6}l^{4m-4}\right)\\[-0.2cm]
&\hspace{-2.8cm}\text{and}\\[-0.2cm]
|r_{n,\lambda,2}(l)|&\le C_{\lambda}\left(\frac{\log^{2m^2+1}n}{n^{m+3/2}}\right).
\end{align*}
\end{theorem}

Notice that for the second error term $r_{n,\gamma,2}$ we have a uniform upper bound for all $l\in \Im(I_{-}-\E I_-)$. Obviously if $36n^{-3}l^2\le \lambda^2\log n$, the leading term on the right hand side of inequality $(\ref{eqn:cla})$ is the dominating one.
Now Theorem \ref{the:cla} enables to calculate asymptotically sharp bounds on the tail probabilities of $\hat\tau(X_{1:n},Y_{1:n})$.
\begin{proposition}\label{pro:taiind}
Under the assumptions of Proposition \ref{pro:keninv} let  $\gamma,\beta>0$ and $p\in\N$, such that $p<n^\beta$. Then,
\begin{align}
\begin{split}
\P\Bigg(|\hat\tau(Y_{1:n},Z_{1:n})|&\ge \gamma \sqrt{\frac{\log p}{n}}\Bigg)\\
&=2\left(1-\Phi\left(\frac{3}{2}\gamma \sqrt{\log p}\right)\right)+R_{n,p,\beta,\gamma},\label{eqn:taiind}
\end{split}
\end{align}
where the error term $R_{n,p,\beta,\gamma}$ satisfies for $m=[\frac{3}{2}+\beta\gamma^2]+1$ and some constant $C_{\gamma,\beta}>0$
\begin{align}
\left|R_{n,p,\beta,\gamma}\right|\le C_{\gamma,\beta}\frac{\log^{2m^2+1}n}{n}p^{-\frac{9}{8}\gamma^2}.\label{eqn:errbou}
\end{align}
\end{proposition}
\begin{proof}
Let $\hat\tau:=\hat\tau(Y_{1:n},Z_{1:n})$, $I_0:=6 n^{-3/2}\left(I_--\E I_-\right)$, $\gamma_{n,p}:=\frac{3}{2}\gamma \sqrt{\log p}-\frac{3}{2}\gamma \frac{\sqrt{\log p}}{n}$ and $\lambda>0$ be first arbitrary, then we have
{\allowdisplaybreaks
\begin{align*}
&\P\left(|\hat\tau|\ge \gamma \sqrt{\frac{\log p}{n}}\right)=\P\left(|I_--\E I_-|\ge \frac{1}{4}\gamma n^{3/2}\sqrt{\log p}-\frac{1}{4}\gamma n^{1/2}\sqrt{\log p}\right)\\
=~&\P\left(|I_0|\ge \frac{3}{2}\gamma \sqrt{\log p}-\frac{3}{2}\gamma \frac{\sqrt{\log p}}{n}  \right)\\
=~&\sum_{\substack{x\in \Im(I_{0}):\\x\ge \gamma_{n,p}}}\sqrt{\frac{2}{n^3\pi}}e^{-\frac{x^2}{2}}+\sum_{\substack{x\in \Im(I_{0}):\\x\ge \gamma_{n,p}}}r_{n,\lambda,1}\left(\frac{n^{\frac{3}{2}}x}{6}\right)+\sum_{\substack{x\in \Im(I_{0}):\\x\ge \gamma_{n,p}}}r_{n,\lambda,2}\left(\frac{n^{\frac{3}{2}}x}{6}\right)\\
=:~&J_1+J_2+J_3.
\end{align*}}

We evaluate $J_1$, $J_2$ and $J_3$ separately. We first give upper bounds on the expressions.

{\it Upper bound on $J_1$:}
{\allowdisplaybreaks \begin{align*}
J_1&\le 2\hspace{-0.5cm}\int\limits_{\gamma_{n,p}-6n^{-3/2}}^\infty\hspace{-0.5cm}\phi(x)dx=2\left(1-\Phi\left(\frac{3}{2}\gamma \sqrt{\log p}\right)\right)+2\hspace{-0.2cm}\int\limits_{\gamma_{n,p}-6n^{-3/2}}^{\frac{3}{2}\gamma \sqrt{\log p}}\hspace{-0.2cm}\phi(x) dx\\
&\le 2\left(1-\Phi\left(\frac{3}{2}\gamma \sqrt{\log p}\right)\right)+C_\gamma \frac{\sqrt{\log p}}{n}\phi \left(\frac{3}{2}\gamma\sqrt{\log p}-\frac{3}{2}\frac{\sqrt{\log p}}{n}-6n^{-3/2}\right)\\
&\le 2\left(1-\Phi\left(\frac{3}{2}\gamma \sqrt{\log p}\right)\right)+C_{\gamma,\beta} \frac{\sqrt{\log p}}{n}\phi \left(\frac{3}{2}\gamma\sqrt{\log p}\right)\\
&= 2\left(1-\Phi\left(\frac{3}{2}\gamma \sqrt{\log p}\right)\right)+C_{\gamma,\beta} \frac{\sqrt{\log p}}{n}p^{-\frac{9}{8}\gamma^2}.
\end{align*}}

{\it Upper bound on $J_2$:} By the error bound in Theorem \ref{the:cla} on $r_{n,\lambda,1}$ and integration by parts we similarly conclude
\begin{align*}
J_2\le C_{\gamma,\beta,m}n^{-1}\int\limits_{\frac{3}{2}\gamma \sqrt{\log p}}^\infty x^{4m-4}\phi(x)dx
\le C_{\gamma,\beta,m}\frac{\log^{4m-3}p}{n}p^{-\frac{9}{8}\gamma^2}.
\end{align*}

{\it Lower bound on $J_3$:} Clearly,
\begin{align*}
J_3\le C_{m}\frac{\log^{2m^2+1}n}{n^{m-\frac{1}{2}}}.
\end{align*}

For an integer $m> 3/2+\frac{9}{8}\beta\gamma^2$ we finally have
\begin{align*}
\P\left(|\hat\tau|\ge \gamma \sqrt{\frac{\log p}{n}}\right)\le  2\left(1-\Phi\left(\frac{3}{2}\gamma \sqrt{\log p}\right)\right)+  C_{\gamma,\beta}\frac{\log^{2m^2+1}n}{n}p^{-\frac{9}{8}\gamma^2}.
\end{align*}

{\it Lower bound on $J_1$:} Analogously to the upper bound on $J_1$ we obtain
\begin{align*}
J_1\ge  2\left(1-\Phi\left(\frac{3}{2}\gamma \sqrt{\log p}\right)\right)-\hspace{-0.2cm}\int\limits_{\frac{3\sqrt{n}}{2}-\frac{3}{2\sqrt{n}}}^\infty\hspace{-0.5cm}\phi(x)dx.
\end{align*}

{\it Lower bounds on $J_2$ and $J_3$:} We have
\begin{align*}
J_2\ge -C_{\gamma,\beta,m}\frac{\log^{4m-3}p}{n}p^{-\frac{9}{8}\gamma^2} \text{ and }J_3\ge C_{m}\frac{\log^{2m^2+1}n}{n^{m-\frac{1}{2}}}.
\end{align*}

Hence, again pick $m> 3/2+\frac{9}{8}\beta\gamma^2$ and derive
\begin{align*}
\P\left(|\hat\tau|\ge \gamma \sqrt{\frac{\log p}{n}}\right)\ge  2\left(1-\Phi\left(\frac{3}{2}\gamma \sqrt{\log p}\right)\right)-  C_{\gamma,\beta}\frac{\log^{2m^2+1}n}{n}p^{-\frac{9}{8}\gamma^2}.
\end{align*}

Combing both bounds provides the desired statement.
\end{proof}

%%%%%%%%%%%%%%%%%%%%%%%%%%%%%%%%%%%%%%%%%%%%%%%%%%%%%%%%%%%%%%%%%%%%%%%%%%%%%%%%%%%%%%%%%%%%%%%%%%%%%%%%%%%%%%%%%%%%%%%%%%%%%%%%%%%%%%%%%%

\subsection{Tails of Kendall's tau sample correlation for normal distributions with weak correlation}

In this subsection we transfer Proposition \ref{pro:taiind} to normal distributed random variables $(Y,Z)$ with weak correlation $\sigma$ and give some conclusions from it. The crucial argument to evaluate the tail probabilities of $\hat\tau(Y_{1:n},Z_{1:n})-\tau(Y,Z)$ for an i.i.d. sample $(Y_1,Z_1),...,(Y_n,Z_n)\overset{D}{=}(Y,Z)$ is to approximate $\hat\tau(Y_{1:n},Z_{1:n})-\tau(Y,Z)$ by Kendall's tau sample correlation for an appropriate sample $(W_1,Z_1),...,(W_n,Z_n)$ with uncorrelated components.~For ease of notations suppose that $Y$ and $Z$ are standardized. Then $Y$ can be written as $Y=\sqrt{1-\sigma^2}W+\sigma Z$ for $(W,Z)\sim \mathcal{N}_2(\mu,\Id)$. This is a natural candidate for the approximation argument.

\begin{lemma}\label{lem:flusma}
Let $(W_1,Z_1),...,(W_n,Z_n)\widesim{i.i.d.} \mathcal{N}_2(\mu,\Id)$, $Y_i=\sqrt{1-\sigma^2}W_i+\sigma Z_i,~i=1,...,n$, where $\sigma^2\le \zeta^2\frac{\log p}{n}\wedge \frac{3}{4}$, $1<n<p\le n^\beta$ for a constant $\beta>1$ and $\zeta>0$. Then for  $c_n= \lambda\frac{n\sqrt{\log p}}{(n-1)n^{1/4}}$, $\lambda>0$, holds
\begin{align*}
\P\left(\left|\hat\tau(Y_{1:n},Z_{1:n})-\tau(Y_1,Z_1)-\hat\tau(W_{1:n},Z_{1:n})\right|\ge c_n\sqrt\frac{\log p}{n}\right)\le 2np^{-C_{\beta,\zeta,\lambda}\sqrt{\log p}}.
\end{align*}
\end{lemma}
\begin{proof}
By $1$-factorization of the complete graph on $n+2\left(\frac{n}{2}-\left[\frac{n}{2} \right]\right)$ vertices and by the union bound we conclude
{\allowdisplaybreaks \begin{align*}
&\P\left(|\hat\tau(Y_{1:n},Z_{1:n})-\tau(Y_1,Z_1)-\hat\tau(W_{1:n},Z_{1:n})|\ge c_n\sqrt{\frac{\log p}{n}}\right) \\
%&=\P\left(\left|\frac{1}{n(n-1)}\sum_{k,l=1}^n\left(\left(\sign(X_k-X_l)-\sign\left(W_k-W_l\right)\right)\sign\left(Y_k-Y_l\right)-\tau(X,Y)\right)\right|\ge c_n\sqrt{\frac{\log p}{n}}\right)\\
=~&\P\Bigg(\Bigg|\frac{2}{n-1}\sum_{\substack{k,l=1 \\ k<l}}^n\left(\left(\sign(Y_k-Y_l)-\sign\left(W_k-W_l\right)\right)\right.\\[-0.55cm]
&\hspace{4.3cm}\times\sign\left.\left(Z_k-Z_l\right)-\tau(Y_1,Z_1)\right)\Bigg|\ge c_n\sqrt{n\log p}\Bigg)\\
\le~& %\left(2n-1-2\left[\frac{n}{2}\right]\right)
n\P\Bigg(\Bigg|\frac{2}{n-1}\sum_{l=1}^{\left[\frac{n}{2}\right]}\Big(\left(\sign(Y_{2k-1}-Y_{2k})-\sign\left(W_{2k-1}-W_{2k}\right)\right)\\[-0.2cm]
&\hspace{3.75cm}\times\sign\left(Z_{2k-1}-Z_{2k}\right)-\tau(Y_1,Z_1)\Big)\Bigg|\ge c_n\sqrt{\frac{\log p}{n}} \Bigg).
%\le n\P\left(\sum_{l=1}^{\left[\frac{n}{2}\right]}\frac{1}{2}\left|\sign(X_{2k-1}-X_{2k})-\sign\left(W_{2k-1}-W_{2k}\right)\right|\ge c_n^\ast\sqrt{n\log p}\right)
\end{align*}}

Now let \begin{align*}\varepsilon_k:=\frac{1}{3}\left(\left(\sign(Y_{2k-1}-Y_{2k})-\sign(W_{2k-1}-W_{2k})\right)\sign(Z_{2k-1}-Z_{2k})-\tau(X,Y)\right).\end{align*}

Obviously, the random variables $\varepsilon_k$ are centered and bounded in absolute value by $1$. We evaluate the variance of $\varepsilon_k$ to apply Bernstein inequality. We have
{\allowdisplaybreaks\begin{align*}
&\Var\left(\varepsilon_k\right)=\E\varepsilon_k^2\le\frac{1}{9}\E\left(\sign\left(Y_{2k-1}-Y_{2k}\right)-\sign\left(W_{2k-1}-W_{2k}\right)\right)^2\\
\le&\P(\sign(Y_{2k-1}-Y_{2k})\neq\sign(W_{2k-1}-W_{2k}))\\
=&\P\left(\sign\left(\sqrt{1-\sigma^2}\left(W_{2k-1}-W_{2k}\right)+\sigma \left(Z_{2k-1}-Z_{2k}\right)\right)\neq  \sign\left(W_{2k-1}-W_{2k}\right)\right)\\
\le&\P\left(\sqrt{1-\sigma^2}\left|W_{2k-1}-W_{2k}\right|<|\sigma| \left|Z_{2k-1}-Z_{2k}\right|\right)\\
\le&\P\left(\left|\frac{W_{2k-1}-W_{2k}}{Z_{2k-1}-Z_{2k}}\right|<2|\sigma|\right)\\
\le&|\sigma|,
\end{align*}}where the last line follows easily from the fact that $\frac{W_{2k-1}-W_{2k}}{Z_{2k-1}-Z_{2k}}$ is standard Cauchy distributed and therefore its density is bounded by $\pi^{-1}$. Finally, we conclude by Bernstein inequality
\begin{align*}
&\P\left(|\hat\tau(Y_{1:n},Z_{1:n})-\tau(Y_1,Z_1)-\hat\tau(W_{1:n},Z_{1:n})|\ge c_n\sqrt{\frac{\log p}{n}}\right)\\ 
&\hspace{5.5cm}\le n\P\Bigg(\Bigg|\sum_{l=1}^{\left[\frac{n}{2}\right]}\varepsilon_k\Bigg|\ge c_n\frac{n-1}{6n} \sqrt{n\log p}\Bigg)\\[0.2cm]
&\hspace{5.5cm}\le2np^{-C_{\beta,\zeta,\lambda}\sqrt{\log p}}.
%&\le n \sum_{k=\left[c_n^\ast\sqrt{n\log p}\right]+1}^{\left[\frac{n}{2}\right]}\binom{\left[\frac{n}{2}\right]}{k}|a|^k(1-|a|)^{n-k}\\
%&\lesssim n  \sum_{k=\left[c_n^\ast\sqrt{n\log p}\right]+1}^{\left[\frac{n}{2}\right]} \frac{n^k|a|^k}{2^kk^{k+1/2}e^{-k}}\\
%&\le n\sum_{k=\left[c_n^\ast\sqrt{n\log p}\right]+1}^{\left[\frac{n}{2}\right]}
\end{align*} 
\end{proof}
\begin{proposition}\label{pro:lardevken}
Let $(Y_1,Z_1),...,(Y_n,Z_n)\widesim{i.i.d.} \mathcal{N}_2(\mu,\Sigma)$, $\Sigma_{11}, \Sigma_{22}>0$, where $|\frac{\Sigma_{12}}{\sqrt{\Sigma_{11}\Sigma_{22}}}|\le \zeta \sqrt{\frac{\log p}{n}}\wedge \frac{3}{4}$ for an arbitrary constant $\zeta>0$ and $n< p \le n^\beta$, $\beta>1$. Then for any real number $\gamma>0$ holds
\begin{align}\label{eqn:conweacor0}
\begin{split}
\P\left(|\hat\tau(Y_{1:n},Z_{1:n})-\tau(Y_1,Z_1)|\right.&\left.\ge \gamma\sqrt{\frac{\log p}{n}}\right)\\
&=2\left(1-\Phi\left(\frac{3}{2}\gamma\sqrt{\log p}\right)\right)+R_{n,p,\beta,\zeta,\gamma},
\end{split}
\end{align}
where the error $R_{n,p,\beta,\zeta,\gamma}$ satisfies for some constant $C_{\beta,\zeta,\gamma}>0$
\begin{align}
|R_{n,p,\beta,\zeta,\gamma}|\le C_{\beta,\zeta,\gamma}\frac{\log p}{n^{1/4}}p^{-\frac{9}{8}\gamma^2}.
\end{align}
\end{proposition}
\begin{proof}
Let $\sigma:=\frac{\Sigma_{12}}{\sqrt{\Sigma_{11}\Sigma_{22}}}$. W.l.o.g. assume that $(W_1,Z_1),...,(W_n,Z_n)\widesim{i.i.d.} \mathcal{N}_2(0,\Id)$ and $Y_i=\sqrt{1-\sigma^2}W_i+\sigma Z_i,~i=1,...,n$. Pick $c_n:=12\frac{n\sqrt{|\sigma|}\sqrt{\log p}}{(n-1)n^{1/4}}$. 

First we give an upper bound on $\P\left(|\hat\tau(Y_{1:n},Z_{1:n})-\tau(Y_1,Z_1)|\ge \gamma\sqrt{\frac{\log p}{n}}\right)$. We have
%\begin{align*}
 %\P\left(\left|\hat\tau(W,Y)\right|\ge \gamma_u\sqrt{\frac{\log p}{n}}\right)\ge\P\left(\left|\hat\tau(X,Y)-\tau(X,Y)-\tau(W,Y)\right|\ge c_n\sqrt\frac{\log p}{n}\right)
%\end{align*}
%\begin{align*}
%\P\left(\left|\hat\tau(W,Y)\right|\ge\gamma_l\sqrt\frac{\log p}{n}\right)\ge 2\P\left(\left|\hat\tau(W,Y)-\hat\tau(X,Y)+\tau(X,Y)\right|\ge c_n\sqrt\frac{\log p}{n}\right).
%\end{align*}
%The latter assumption is fulfilled for large $n$ due to Proposition \ref{pro:taiind} and Lemma \ref{lem:flusma}. We conclude 
\allowdisplaybreaks{\begin{align*}
&\P\left(|\hat\tau(Y_{1:n},Z_{1:n})-\tau(Y_1,Z_1)|\ge \gamma\sqrt{\frac{\log p}{n}}\right)\\
\le~&\P\left(\left|\hat\tau(Y_{1:n},Z_{1:n})-\tau(Y_1,Z_1)-\hat\tau(W_{1:n},Z_{1:n})\right|+|\hat\tau(W_{1:n},Z_{1:n})|\ge \gamma\sqrt{\frac{\log p}{n}}\right)\\
\le~&\P\left(\left|\hat\tau(W_{1:n},Z_{1:n})\right|\ge (\gamma-c_n)\sqrt{\frac{\log p}{n}}\right)\\
&\hspace{2.25cm}+\P\left(\left|\hat\tau(Y_{1:n},Z_{1:n})-\tau(Y_1,Z_1)-\hat\tau(W_{1:n},Z_{1:n})\right|\ge c_n\sqrt\frac{\log p}{n}\right).
%%&\hspace{2.5cm}\le 2\P\left(\left|\hat\tau(W,Y)\right|\ge \gamma_u\sqrt{\frac{\log p}{n}}\right).
\end{align*}}

The second summand is easily handled by Lemma $\ref{lem:flusma}$. For the first summand notice that $\gamma-c_n$ is bounded above uniformly for all $n$. Therefore we apply equation $(\ref{eqn:taiind})$, where the constant in the error bound $(\ref{eqn:errbou})$ can be chosen independently from $n$, such that
\begin{align*}
|R_{n,p,\beta,\gamma-c_n}|&\le C_{\gamma,\beta,\zeta}\frac{\log^{2m^2+1}n}{n}p^{-\frac{9}{8}(\gamma-c_n)^2}\le C_{\gamma,\beta,\zeta}\frac{\log^{2m^2+1}n}{n}p^{-\frac{9}{8}\gamma^2}.
\end{align*}
Hence, equation $(\ref{eqn:taiind})$ yields
\begin{align*}
P\Bigg(\left|\hat\tau(W_{1:n},Z_{1:n})\right|&\ge(\gamma-c_n)\sqrt{\frac{\log p}{n}}\Bigg)\\
&\hspace{-0.6cm}\le 2\left(1-\Phi\left(\frac{3}{2}(\gamma-c_n) \sqrt{\log p}\right)\right)+C_{\gamma,\beta,\zeta}\frac{\log^{2m^2+1}n}{n}p^{-\frac{9}{8}\gamma^2}\\
&\hspace{-0.6cm}\le 2\left(1-\Phi\left(\frac{3}{2}\gamma \sqrt{\log p}\right)\right)+C_{\gamma,\beta,\zeta}\frac{\log p}{n^{1/4}}p^{-\frac{9}{8}\gamma^2}.
\end{align*}
This provides the upper bound.
The lower bound arises from the following computation, where we finally apply Proposition \ref{pro:taiind} and Lemma $\ref{lem:flusma}$ again:
{\allowdisplaybreaks
\begin{align*}
&\P\left(|\hat\tau(Y_{1:n},Z_{1:n})-\tau(Y_1,Z_1)|\ge \gamma\sqrt{\frac{\log p}{n}}\right)\\
\ge~&\P\left(\left|\hat\tau(W_{1:n},Z_{1:n})\right|-\left|\hat\tau(W_{1:n},Z_{1:n})-\hat\tau(Y_{1:n},Z_{1:n})+\tau(Y_,Z_1)\right|\ge \gamma \sqrt{\frac{\log p}{n}}\right)\\
\ge~&\P\Bigg(\left|\hat\tau(W_{1:n},Z_{1:n})\right|-\left|\hat\tau(W_{1:n},Z_{1:n})-\hat\tau(Y_{1:n},Z_{1:n})+\tau(Y_1,Z_1)\right|\ge \gamma \sqrt{\frac{\log p}{n}},\\
&\hspace{3cm}\left|\hat\tau(W_{1:n},Z_{1:n})-\hat\tau(Y_{1:n},Z_{1:n})+\tau(Y_1,Z_1)\right|<c_n\sqrt\frac{\log p}{n}\Bigg)\\
\ge~&\P\left(\left|\hat\tau(W_{1:n},Z_{1:n})\right|\ge(\gamma+c_n)\sqrt\frac{\log p}{n}\right)\\
&\hspace{2cm}-\P\left(\left|\hat\tau(W_{1:n},Z_{1:n})-\hat\tau(Y_{1:n},Z_{1:n})+\tau(Y_1,Z_1)\right|\ge c_n\sqrt\frac{\log p}{n}\right).
%&\hspace{0.5cm}\ge \frac{1}{2}\P\left(\left|\hat\tau(W,Y)\right|\ge\gamma_l\sqrt\frac{\log p}{n}\right).
\end{align*}}
%Let $W_k:=\sqrt{1-a^2}X_k+aY_k$ for all $k=1,...,n$. Then we hav 
\end{proof}

We close this section with two straightforward consequences from the last proposition.
\begin{corollary}\label{lem:conweacor}
Under the assumptions of Proposition \ref{pro:lardevken} let $\gamma>0$. Then,
\begin{align}\label{eqn:conweacor}
\begin{split}
\tilde C_{\beta,\zeta,\gamma}\Phi\left(\frac{3}{2}\gamma\sqrt{\log p}\right)&\le\P\left(|\hat\tau(Y_{1:n},Z_{1:n})-\tau(Y_1,Z_1)|\ge \gamma\sqrt{\frac{\log p}{n}}\right)\\
&\le C_{\beta,\zeta,\gamma}\Phi\left(\frac{3}{2}\gamma\sqrt{\log p}\right),
\end{split}
\end{align}
where the constants $\tilde C_{\beta,\zeta,\gamma},C_{\beta,\zeta,\gamma}>0$ depend on $\beta,~ \zeta$ and $\gamma$.
\end{corollary}

In the final corollary the sample $(Y_1,Z_1),...,(Y_n,Z_n)$ and the quantities $\mu,\Sigma$ depend on $n$ even if it is not apparent from the notation.
\begin{corollary}\label{cor:asy}
Let $(Y_1,Z_1),...,(Y_n,Z_n)\widesim{i.i.d.} \mathcal{N}_2(\mu,\Sigma,\xi)$, $\Sigma_{11}, \Sigma_{22}>0$, where $\left|\frac{\Sigma_{12}}{\sqrt{\Sigma_{11}\Sigma_{12}}}\right|\le \zeta \sqrt{\frac{\log p}{n}}\wedge \frac{3}{4}$ for an arbitrary constant $\zeta>0$ and $n< p \le n^\beta$, $\beta>1$. Then for any real number $\gamma>0$ holds
\begin{align}\label{eqn:conweacor2}
\underset{n\to\infty}{\lim}\frac{\P\left(|\hat\tau(Y_{1:n},Z_{1:n})-\tau(Y_1,Z_1)|\ge \gamma\sqrt{\frac{\log p}{n}}\right)}{2\left(1-\Phi\left(\frac{3}{2}\gamma\sqrt{\log p}\right)\right)}=1.
\end{align}
\end{corollary}

%%%%%%%%%%%%%%%%%%%%%% NEW SECTION	%%%%%%%%%%%%%%%%%%%

\section{Summary and discussion}\label{sec:dis}

In this article we have studied the question how much information an entrywise hard threshold matrix estimator is allowed to keep from the pilot estimate $\sin[{\frac{\pi}{2}\hat\tau}]$ to obtain an adaptive rate optimal estimator for a sparse correlation matrix. It is shown that $\alpha^\ast(\frac{\log{p}}{n})^{1/2},~\alpha^\ast=\frac{2\sqrt{2}}{3},$ is a critical threshold level on the entries of $\hat\tau$ for Gaussian observations. This means that any threshold constant $\alpha>\alpha^\ast$ provides an adaptive minimax estimator whereas for $\alpha<\alpha^\ast$ the threshold estimator $\hat\rho^\ast$ does not achieve the optimal rate over sparsity classes $\mathcal{G}_{\text{w},q}(c_{n,p})$ without sufficiently dense correlation matrices. It is not clear how to prove analogous statements for broader classes of elliptical distributions since even the asymptotic variance of the entries $\sin(\frac{\pi}{2}\hat\tau_{ij})$ does not only depend on $\Sigma$ but on $\xi$ as well - see \cite{Lehmann1998}. However, the critical threshold constant for elliptical distributions is at most by factor $\frac{3}{2^{1/2}}$ worse compared to the Gaussian model.
 
%Threshold estimators with small threshold constants are meaningful because of two essential reasons. On the one hand the smaller the threshold constant is the more dependency structure is captured by the estimator, on the other hand small threshold constants are of practical interest. To discuss the second reason let us suppose that we have a threshold level $10(\frac{\log p}{n})^{1/2}$  for instance. Then, we already need a sample size larger than hundred just to compensate the threshold constant $10$, where we have not even considered the dimension $p$ which is typically at least in the thousands. Thus, in practice moderate threshold constants may lead to a trivial estimator which provides the identity as an estimate for the correlation matrix, no matter what we actually observe.\\
In general, the proposed estimators $\hat\rho^\ast$ and $\hat\rho$ do not necessarily need to be positive semi-definite. Obviously, replacing the estimator, say $\hat\rho$, by its projection\\[-0.3cm]
\[
\hat C\in\arg \min_{\substack{C\in\mathcal{S}_+^p\\ C_{ii}=1 \forall i}} \|\hat\rho-C\|_F^2\vspace{-0.1cm}
\]
on the set of all correlation matrices does not affect the minimax rate under the Frobenius norm loss. The numerical literature offers several algorithms to compute the nearest correlation matrix in the above sense. The reader is referred to \cite{Borsdorf} for an overview of  this issue, the current state of the art and MATLAB implementations.

It is not clear which rate the threshold estimator attains for $\alpha=\alpha^\ast$ since Lemma $\ref{lem:concor}$ is not applicable for that threshold constant. This case is important because in the proof of Theorem  \ref{the:optuppbou2} the constant $C_{\alpha,\eta_u}$ in the upper bound on the maximal risk of $\hat\rho_{\alpha}$  tends to infinity as $\alpha\downarrow\alpha^\ast$. So, if $\hat\rho_{\alpha^\ast}$ attains the minimax rate, the constants $C_{\alpha,\eta_u}$ in Theorem  \ref{the:optuppbou} should be substantially improvable. The critical threshold constant is so far restricted to minimax estimation under the Frobenius norm loss. By Theorem \ref{optspe} the critical threshold constant under the spectral norm lies within $[\frac{2\sqrt{2}}{3},\frac{4}{3}]$ in the Gaussian model. For the exact critical threshold constant under the spectral norm one needs an appropriate upper bound on the expectation of the squared $\ell_1$-norm of the adjacency matrix $\hat M=(\hat M_{ij})_{i,j=1,...,p}$ with\\[-0.3cm] $$\hat M_{ij}:=\ind\left(|\hat\tau_{ij}^\ast-\tau_{ij}|>\beta\min\left(\tau_{ij},\alpha\sqrt\frac{\log p}{n}\right)\right)\vspace{-0.1cm}$$
for a sufficiently large value $\beta>0$ and $\tau\in\G_{\text{w},q}(c_{n,p})$. However, a solution to this task seems currently out of reach.

%A further open problem is the identification of the threshold estimator $\hat\rho_\alpha$ with the asymptotically smallest maximal risk. At least under the Frobenius norm %Corollary \ref{cor:asy}  should enable to compute the exact asymptotic constant of the maximal risk for any minimax optimal estimator $\hat\rho=\hat\rho_\alpha$ under slight %regularization, i.e. we have to evaluate the limit\\[-0.3cm]
%\[
%\limsup_{n\to\infty}\underset{\rho\in\mathcal{G}_{\text{w},q}(c_{n,p})}{\sup}c_{n,p}^{-1}\left(\frac{n}{\log p}\right)^{1-q/2}p^{-1}\E\|\rho-\hat\rho\|_F^2.\vspace{-0.1cm}
%\]
%Minimizing over all $\alpha$ provides the threshold estimator with the asymptotically smallest maximal loss.

As explained in the introduction the results of the paper may be extended to the estimation of latent generalized correlation matrices in much broader families of distributions. Nevertheless, it is also important to identify further models, where weak Kendall's tau correlation implies weak Pearson's correlation. In such models sparsity masks (\cite{Levina2012}) may be determined for covariance and correlation matrix estimation based on Kendall's tau sample correlation. This enables to avoid troubles with the tails of the underlying distribution.

%%%%%%%%%%%%%%%%%%%%%% NEW SECTION	%%%%%%%%%%%%%%%%%%%

\section{Proofs of Section 3 and 4}\label{sec:pro}
We start by the proof of Theorem \ref{the:minlowbouken}. The lower bound for estimating $\rho$ over some class $\G_{\text{w},q}(c_{n,p})$ with radius $c_{n,p}>2v^q(\frac{\log p}{n})^{q/2}$ is an immediate consequence of the proof of Theorem 4 in the article \cite{Caib} as already mentioned in Section \ref{sec:min}. So, to complete the proof of Theorem \ref{the:minlowbouken} it remains to restrict to the case $c_{n,p}\le 3(\frac{\log p}{n})^{q/2}$. 
Therefore, we need some additional notation to restate the minimax lower bound technique developed recently by \cite{Caib}.\\ 
For a finite set $B\subset\R^p\setminus\{0\}$ let $\Lambda\subset B^r$. Then, define the parameter space
\begin{align}
\Theta:=\Gamma\otimes\Lambda=\{\theta=(\gamma,\lambda):\gamma\in\Gamma=\{0,1\}^r \text{ and }\lambda\in\Lambda\subset B^r\}.
\end{align}
Rewrite $\theta\in \Theta$ by the representation $\theta=(\gamma(\theta),\lambda(\theta)),~\gamma(\theta)\in\Gamma,\lambda(\theta)\in\Lambda$. The $i$-th coordinate of $\gamma(\theta)$ is denoted by $\gamma_i(\theta)$ and $\Theta_{i,a}:=\{\theta\in\Theta:\gamma_i(\theta)=a\}$ is the set of all parameters $\theta$ such the $i$-coordinate of $\gamma_i(\theta)$ is fixed by the value $a\in\{0,1\}$. Finally let 
\begin{align}
H(\theta,\theta')=\sum_{i=1}^r|\gamma_i(\theta)-\gamma_i(\theta')|
\end{align}
be the Hamming distance on $\Theta$ and 
\begin{align}
\|\P\wedge \Q\|=\int p\wedge q~d\mu
\end{align}
the total variation affinity, where $\P$ and $\Q$ have densities $p$ and $q$ with respect to a common dominating measure $\mu$.
\begin{lemma}[\cite{Caib}]\label{lem:mintec}
For any metric $d$ on $\Theta$, any positive real number $s$ and any estimator $T$ of $\psi(\theta)$ based on an observation $\mathbf{X}$ from the experiment $\{\P_{\theta},\theta\in\Theta\}$ holds
\begin{align}
\underset{\Theta}{\max}2^s\E d^s(T,\psi(\theta))\ge\underset{\{(\theta,\theta'):H(\gamma(\theta),\gamma(\theta'))\ge 1\}}{\min}\frac{d^s(\psi(\theta),\psi(\theta'))}{H(\gamma(\theta),\gamma(\theta'))}\frac{r}{2}\underset{1\le i\le r}{\min}\|\bar\P_{i,0}\wedge \bar\P_{i,1}\|,
\end{align}
where the mixture distribution $\bar\P_{i,a}$, $a\in\{0,1\}$ is defined by
\begin{align}
\bar\P_{i,a}=\frac{1}{2^{r-1}|\Lambda|}\sum\limits_{\theta\in \Theta_{i,a}}\P_\theta.
\end{align}
\end{lemma}

In contrast to the usual techniques for establishing minimax lower bounds as LeCam's method and Assouad's lemma, Lemma \ref{lem:mintec} enables to handle two-directional problems. In the context of sparse correlation matrix estimation this means that first one needs to recognize non-zero rows of the population correlation matrix to identify the non-zero entries on each non-zero row afterwards.\\
Before we prove Theorem \ref{the:minlowbouken}, let us mention the following useful Lemma to evaluate the chi-squared squared distance between Gaussian mixtures.
\begin{lemma}\label{lem: Thomas}
Let $\Sigma_i\in \R^{p\times p},~i=0,\dots,2,$ be positive definite covariance matrices such that $\Sigma_1^{-1}+\Sigma_2^{-1}-\Sigma_0^{-1}$ is positive definite, and $g_i,~i=0,\dots,2,$ be the density function of $\mathcal{N}_p(0,\Sigma_i)$. Then,
\begin{align*}
\int \frac{g_1g_2}{g_0}d\lambda=\left[\det\left(\Sigma_0^{-1}\left(\Sigma_2+\Sigma_1-\Sigma_2\Sigma_0^{-1}\Sigma_1\right)\right)\right]^{-1/2}.
\end{align*}
\end{lemma}
The proof of the lemma is a straightforward calculation and therefore omitted.
\begin{proof}[Proof of Theorem \ref{the:minlowbouken}]
Define a finite subset $\mathcal F^\ast\subset\G_{\text{w},q}(c_{n,p})$ as follows: Let $r=\left[\frac{p}{2}\right]$, $B=\{e_i~|~i=\left[\frac{p}{2}\right]+1,...,p\}$, where $e_i$ denotes the $i$-th canonical basis vector of $\R^p$, and $\Lambda\subset B^r$ such that $\lambda=(b_1,...,b_r)\in\Lambda$ iff $b_1,...,b_r\in B$ are distinct. Then, let $\mathcal F^\ast$ be the set of all matrices of the form 
\begin{align}
\Sigma(\theta)=\Id+\varepsilon_{n,p}\sum\limits_{j=1}^r\gamma_i(\theta)A(\lambda_j(\theta)),~\theta\in\Theta=\Gamma\otimes \Lambda,
\end{align}
where $$\varepsilon_{n,p}=\frac{1}{3}\left(\frac{c_{n,p}}{3\vee M\log 3 }\right)^{1/q}\wedge v\sqrt{\frac{\log p}{n}} \ \ \text{with} \ \ v\le \sqrt\frac{\log 3-\log (5/2)}{2\log 3}.$$ Here, the $j$-th row of the $p\times p$ symmetric matrix $A(\lambda_j(\theta))$ is given by $\lambda_j(\theta)$ and the submatrix of  $A(\lambda_j(\theta))$ resulting by deleting the $j$-th row and the $j$-th column is the $(p-1)\times (p-1)$ null matrix. Note that by Assumption $(A_3)$ and $c_{n,p}\le 3(\log p/n)^{q/2}$ it holds $\varepsilon_{n,p}^2\le \frac{1}{9}$. By construction $\Sigma(\theta)$ has at most one non-zero entry per row off the diagonal. Obviously, $\mathcal F^\ast\subset \G_{\text{w},q}(c_{n,p})$. Now we apply Lemma \ref{lem:mintec} on all centered normal distributions with covariance matrix $\Sigma\in\mathcal{F}^\ast$ and obtain for some constants $C_{M,q},\tilde C_{m,q}>0$
\begin{align}
\underset{\Theta}{\max}~\E_{\mathbf{X}|\theta}\|T-\Sigma(\theta)\|_F^2\ge C_{M,q}r \left(c_{n,p}^{2/q}\wedge \frac{\log p}{n}\right)\underset{1\le i\le r}{\min}\|\bar\P_{i,0}\wedge \bar\P_{i,1}\|
\end{align}
and
\begin{align}
\underset{\Theta}{\max}~\E_{\mathbf{X}|\theta}\|T-\Sigma(\theta)\|_2^2\ge C_{M,q}r \frac{c_{n,p}^{2/q}\wedge \frac{\log p}{n}}{p}\underset{1\le i\le r}{\min}\|\bar\P_{i,0}\wedge \bar\P_{i,1}\|
\end{align}
It remains to prove that $\underset{1\le i\le r}{\min}\|\bar\P_{i,0}\wedge \bar\P_{i,1}\|$ is bounded away from zero. Analogously to the proof of Lemma 6 in \cite{Caib}, it is sufficient to show that
\begin{align}
\tilde\E_{\gamma_{-1},\lambda_{-1}}\left(\int\left(\frac{d\bar\P_{1,1,\gamma_{-1},\lambda_{-1}}}{d\bar\P_{1,0,\gamma_{-1},\lambda_{-1}}}\right)^2d\bar\P_{1,0,\gamma_{-1},\lambda_{-1}}-1\right)\le c_2^2
\end{align}
for some constant $c_2<1$, where
\begin{align}
\begin{split}
&\bar\P_{i,a,b,c}:=\frac{1}{|\Theta_{i,a,b,c}|}\sum\limits_{\theta\in\Theta_{(i,a,b,c)}}\P_\theta,\\
&\Theta_{(i,a,b,c)}:=\{\theta\in\Theta:\gamma_i(\theta)=a,\gamma_{-i}(\theta)=b, \lambda_{-i}(\theta)=c\},
\end{split}
\end{align}
and for any function $h$
\begin{align}
\begin{split}
&\tilde\E_{\gamma_{-i},\lambda_{-i}}h(\gamma_{-i},\lambda_{-i})=\sum\limits_{(b,c)}\frac{|\Theta_{(i,a,b,c)}|}{|\Theta_{i,a}|}h(b,c)\\
&\Theta_{-i}:=\{(b,c)|\exists\theta\in\Theta:\gamma_{-i}(\theta)=b \wedge \lambda_{-i}(\theta)=c\}
\end{split}
\end{align}
is the average of $h(b,c)$ over the set $\Theta_{-i}$. Since $\bar\P_{1,0,b,c}$ is just the joint distribution of $n$ i.i.d Gaussian vectors $X_1,...,X_n\in\R^p$ we conclude
\begin{align}
\begin{split}
&\int\left(\frac{d\bar\P_{1,1,\gamma_{-1},\lambda_{-1}}}{d\bar\P_{1,0,\gamma_{-1},\lambda_{-1}}}\right)^2d\bar\P_{1,0,\gamma_{-1},\lambda_{-1}}-1\\[0.1cm]
&\hspace{1.6cm}=\frac{1}{|\Theta_{1,1,\gamma_{-1},\lambda_{-1}}|^2}\sum_{\theta,\theta'\in\Theta_{(1,1, \gamma_{-1},\lambda_{-1})}}\int\frac{d\P_\theta d\P_{\theta'}}{d\P_{((0,\gamma_{-1}),(b,\lambda_{-1}))}}-1
\end{split}
\end{align}
for an arbitrary permissible vector $b$. Moreover by Lemma \ref{lem: Thomas}
\begin{align}\label{eqn:chidis}
%\begin{split}
&\hspace{-3cm}\int\left(\frac{d\bar\P_{1,1,\gamma_{-1},\lambda_{-1}}}{d\bar\P_{1,0,\gamma_{-1},\lambda_{-1}}}\right)^2d\bar\P_{1,0,\gamma_{-1},\lambda_{-1}}-1\\
\le\frac{1}{\left(\left[\frac{p}{2}\right]+1\right)^2}\sum_{\theta,\theta'\in\Theta_{(1,1, \gamma_{-1},\lambda_{-1})}}&\hspace{-1cm}\left[\det\left(\Sigma(\theta_0)^{-1}\left(\Sigma(\theta')+\Sigma(\theta)-\Sigma(\theta')\Sigma(\theta_0)^{-1}\Sigma(\theta)\right)\right)\right]^{-\frac{n}{2}}-1,\notag
%\end{split}
\end{align}
where $\theta_0=((0,\gamma_{-1}),(b,\lambda_{-1}))$. It remains to evaluate $$\det\left(\Sigma(\theta_0)^{-1}\left(\Sigma(\theta')+\Sigma(\theta)-\Sigma(\theta')\Sigma(\theta_0)^{-1}\Sigma(\theta)\right)\right).$$ First consider the case $\theta\neq\theta'$. Then it holds for some $[p/2]+1\le i\neq j\le p$ that $\lambda_1(\theta)=e_i$ and $\lambda_1(\theta')=e_j$. Rewrite $\Sigma(\theta)=\Sigma(\theta_0)+\varepsilon_{n,p}A(e_i)$ and $\Sigma(\theta')=\Sigma(\theta_0)+\varepsilon_{n,p}A(e_j)$, and note that the $i$-th and $j$-th row as well as column of $\Sigma_0^{-1}$ are equal to $e_i$ and $e_j$. Then, one easily verifies
$$\Sigma(\theta')+\Sigma(\theta)-\Sigma(\theta')\Sigma(\theta_0)^{-1}\Sigma(\theta)=\Sigma(\theta_0)-\varepsilon_{n,p}^2e_je_i^T.$$
Hence,
\begin{align}\label{eqn:det1}
\det\left(\Sigma(\theta_0)^{-1}\left(\Sigma(\theta')+\Sigma(\theta)-\Sigma(\theta')\Sigma(\theta_0)^{-1}\Sigma(\theta)\right)\right)=1.
\end{align}
%Then we have for some pair $(i,j)$, $1\le i\neq j \le p$,
%\begin{align}
%(\Sigma(\theta)-\Sigma(\theta_0))(\Sigma(\theta')-\Sigma(\theta_0))_{kh}=\begin{cases}
%0 &\text{ if }(k,h)\neq (i,j)\\
%\varepsilon_{n,p}^2 &\text{ if }(k,h)=(i,j).
%\end{cases}
%\end{align}
%One easily verifies that the $j$-th column and row vector of $\Sigma(\theta_0)^{-2}$ are equal to the $j$-th unit vector and therefore
%\begin{align}\label{eqn:det1}
%\det(\Id-\Sigma(\theta_0)^{-2}(\Sigma(\theta)-\Sigma(\theta_0))(\Sigma(\theta')-\Sigma(\theta_0)))=1.
%\end{align}
For $\theta=\theta'$ we conclude that
$$2\Sigma(\theta)-\Sigma(\theta)\Sigma(\theta_0)^{-1}\Sigma(\theta)=\Sigma(\theta_0)-\varepsilon_{n,p}^2e_1e_1^T-\varepsilon_{n,p}^2e_ie_i^T$$
for some $[p/2]+1\le i \le p$ since the first and $i$-th row as well as column of $\Sigma(\theta_0)^{-1}$ are equal to $e_1$ and $e_i$. So,
\begin{align}\label{eqn:det2}
\det\left(\Sigma(\theta_0)^{-1}\left(2\Sigma(\theta)-\Sigma(\theta)\Sigma(\theta_0)^{-1}\Sigma(\theta)\right)\right)=(1-\varepsilon_{n,p}^2)^2.
\end{align}
Plugging  (\ref{eqn:det1}) and (\ref{eqn:det2}) into (\ref{eqn:chidis}) yields
\begin{align*}
\int\left(\frac{d\bar\P_{1,1,\gamma_{-1},\lambda_{-1}}}{d\bar\P_{1,0,\gamma_{-1},\lambda_{-1}}}\right)^2d\bar\P_{1,0,\gamma_{-1},\lambda_{-1}}-1\le \frac{2}{p}(1-\varepsilon_{n,p}^2)^{-n}.
%&\le \frac{2}{p}\exp\left(-n\log(1-\varepsilon_{n,p})\right)\\
%&\le \frac{4}{p}\exp\left(n\frac{\varepsilon_{n,p}^2}{1-\varepsilon_{n,p}^2}\right)\\
%&\le \frac{2}{p}p^{\frac{\nu}{1-m}}\\
%\le \frac{2}{p^{1-2\nu^2}}\le \frac{4}{5}.
\end{align*}
By the elementary bound
$$\frac{1}{1-x}\le \exp\left(\frac{x}{1-x}\right),~0\le x <1,$$
we finally conclude
\begin{align*}
\int\left(\frac{d\bar\P_{1,1,\gamma_{-1},\lambda_{-1}}}{d\bar\P_{1,0,\gamma_{-1},\lambda_{-1}}}\right)^2d\bar\P_{1,0,\gamma_{-1},\lambda_{-1}}-1&\le \frac{2}{p}\exp\left(\frac{n\varepsilon_{n,p}^2}{1-\varepsilon_{n,p}^2}\right)\\
%&\le \frac{2}{p}p^{\frac{\nu}{1-m}}\\
%&\le \frac{2}{p}\exp\left(\frac{4n\varepsilon_{n,p}^2}{3}\right)\\
%&\le \frac{2}{p^{1-2\nu^2}}\\
&\le \frac{2}{p^{1-2\nu^2}}\le \frac{4}{5}.
\end{align*}
%Now we prove inequality (\ref{eqn:minlowsam2}). It is sufficient to restrict to estimators such that $|\hat\tau_{ij}|\le 1$ for all $i,j=1,...,p$. Therefore let $\hat\tau$ be an %estimator of Kendall's tau sample correlation matrix, then $\sin[\frac{\pi}{2}\hat\tau]$ is an estimator of the correlation matrix $\rho$. Moreover, by Lemma \ref{lem:spatra} $
%\sin[\frac{\pi}{2}\tau]\in\G_{\text{w},q}(c_{n,p})$ implies $\tau\in\G_{\text{w},q}(c_{n,p})$. We conclude by Lemma \ref{lem:boukencor} and by the first part of Theorem
%\ref{the:minlowboucor}
%\begin{align*}
%\underset{\hat \tau}{\inf} \underset{\tau\in \G_{\text{w},q}(c_{n,p})}{\sup}\E\|\hat\tau-\tau\|_F^2&\ge \frac{4}{\pi^2}\underset{\hat \tau}{\inf} \underset{\tau\in \G_{\text{w},q}
%(c_{n,p})}{\sup}\E\|\sin[\frac{\pi}{2}\hat\tau]-\sin[\frac{\pi}{2}\tau]\|_F^2\\
%&\ge \frac{4}{\pi^2}\underset{\hat \tau}{\inf} \underset{\sin[\frac{\pi}{2}\tau]\in \G_{\text{w},q}(c_{n,p})}{\sup}\E\|\sin[\frac{\pi}{2}\hat\tau]-\sin[\frac{\pi}{2}\tau]\|_F^2\\
%&=\frac{4}{\pi^2}\underset{\hat\tau}{\inf} \underset{\rho\in \G_{\text{w},q}(c_{n,p})}{\sup}\E\|\sin[\frac{\pi}{2}\hat\tau]-\rho\|_F^2\\
%&\ge \frac{4}{\pi^2}\underset{\hat\rho}{\inf} \underset{\rho\in \G_{\text{w},q}(c_{n,p})}{\sup}\E\|\hat\rho-\rho\|_F^2\\
%&\ge Cc_{n,p}\left(c_{n,p}^{2/q}\wedge\frac{\log p}{n}\right)^{1-\frac{q}{2}}.
%\end{align*}
\end{proof}
%\exp(\log(3)*\log(2,5)/\log(3))
%$$\frac{2}{3^{1-2\nu^2}}=\frac{4}{5}$$
%$$(1-2\nu^2)=\frac{\log \frac{5}{2}}{\log 3}$$
%$$\nu=\sqrt{\frac{1}{2}-\frac{\log \frac{5}{2}}{2\log 3}}$$
\begin{proof}[Proof of Proposition \ref{pro:ult}]
%Obviously both statements in line (\ref{eqn:ult2}) are equivalent.  Therefore we only prove the first one. In case of $q=0$ the set $\G_{\text{w},q}(c_{n,p}) =\{\Id\}$ for sufficiently large sample size $n$ and hence there is nothing to prove. So assume that $0<q<1$. 
For $j=1,...,p-1$ denote
\begin{align*}
(j)_{i}:=\begin{cases}[j]_{\tau_i}& \text{for }[j]_{\tau_i}<i,\\
[j]_{\tau_i}+1& \text{else.}
\end{cases}
\end{align*}
We have
\begin{align*}
p^{-1}\|\Id-\rho\|^2_F&=p^{-1}\sum_{i=1}^p\sum_{j=1}^{p-1}|\tau_{i(j)_{i}}|^2\le c_{n,p}^{2/q}\left(1+\sum_{j=2}^p j^{-2/q}\right)\\
&\le c_{n,p}^{2/q}\left(1+\int\limits_{1}^\infty x^{-2/q}dx\right)= \left(1+\frac{q}{2-q}\right)c_{n,p}^{2/q}.
%&\le 2c_{n,p}\left(\frac{\log p}{n}\right)^{1-q/2}.
\end{align*}
Moreover, since 
$$\| \Id-\rho \|^2_2\le \| \Id-\rho \|^2_1 = \max_{i} \left(\sum_{j=1}^{p-1}|\tau_{i(j)_{i}}|\right)^2,$$
we conclude
$$\| \Id-\rho \|^2_2\le c_{n,p}^{2/q}\left(1+\int\limits_{1}^\infty x^{-1/q}dx\right)^2\le \left(1+\frac{q}{1-q}\right)^2c_{n,p}^{2/q}.$$
\end{proof}

Next we prove Theorem \ref{the:minuppbouken}. Therefore we need the following lemma, which is an advancement of Lemma 8 of \cite{Caic} for our purpose of minimizing the necessary threshold constant in Theorem \ref{the:minuppbouken}.
\begin{lemma}\label{lem:lareve}
Let $(Y_1,Z_1),...,(Y_n,Z_n)\in\R^2$ be i.i.d. random variables and for any $\alpha>2$ and $\beta>\frac{\alpha+2}{\alpha-2}$ let the event
\begin{align*}
B:=\left\{|\hat\tau^\ast(Y_{1:n},Z_{1:n})-\tau(Y_1,Z_1)|\le \beta \min\left(|\tau(Y_1,Z_1)|, ~\alpha\sqrt{\frac{\log p}{n}}\right)\right\}.
\end{align*}
Then we have 
\begin{align*}
\P(B)\ge 1-4p^{-\left(1+\varepsilon\right)},
\end{align*}
where $\varepsilon:=\frac{\alpha^2-4}{4}-\alpha^2\frac{(\beta-1)^2-(\beta+1)^2}{4(\beta+1)^2}>0$.
\end{lemma}
 \begin{proof}
\noindent Let $\hat\tau:=\hat\tau(Y_{1:n},Z_{1:n})$, $\hat\tau^\ast:=\hat\tau^\ast(Y_{1:n},Z_{1:n})$ and $\tau:=\tau(Y_1,Z_1)$. First note that by Hoeffding's inequality for U-statistics (see \cite{Hoeffding1963}) 
\begin{align}\label{eqn:Hoeine}
\P\left(|\hat\tau-\tau|>t\right)\le2\exp\left(-\frac{nt^2}{4}\right) \text{ for any }t>0.
\end{align}

We distinguish three cases:
\begin{itemize}
\item[(i)] $|\tau|\le \frac{2\alpha }{\beta+1}\sqrt{\frac{\log p}{n}}$: Since on the event $A:=\left\{|\hat\tau|<\alpha\sqrt\frac{\log p}{n}\right\}$ holds $|\hat\tau^\ast-\tau|=|\tau|\le \beta\min\left(|\tau|,\alpha\sqrt{\frac{\log p}{n}}\right)$, we have $A\subset B$. Now by the triangle inequality and inequality (\ref{eqn:Hoeine}) follows
\begin{align*}
\P(B)
%\P\left(\left|\tilde{\Sigma}_{ij}-\Sigma_{ij}\right|\le 4\min\left\{\sqrt{\Sigma_{ii}^{\ast}\Sigma_{jj}^{\ast}},~\gamma\sqrt{\frac{\log p}{n}}\right\}\right)
&
\ge\P(A)
%\ge \P\left(|\hat{\Sigma}_{ij}|\le\gamma\sqrt{\frac{\log p}{n}}\right)\\
%&= \P\left(\hat{\Sigma}_{ii}-\Sigma_{ii}^{\ast}\le\sigma^2+\gamma\sqrt{\frac{\log p}{n}}-\Sigma_{ii}^{\ast}\right)
\ge \P\left(\left|\hat\tau-\tau\right|< \frac{\beta-1}{\beta+1}\alpha \sqrt{\frac{\log p}{n}}\right)\\
&\ge 1-2p^{-\frac{(\beta-1)^2\alpha^2}{4(\beta+1)^2}}.
\end{align*}
\item[(ii)] $\frac{2\alpha}{\beta+1} \sqrt{\frac{\log p}{n}}<|\tau|< \frac{2\alpha\beta}{\beta+1} \sqrt{\frac{\log p}{n}}$: On $A$ we have again  $|\hat\tau^\ast-\tau|=|\tau|\le \beta\min\left(|\tau|,\alpha\sqrt{\frac{\log p}{n}}\right)$. This implies $A\subset B$. Furthermore consider the event $C:=A^c\cap \left\{|\hat\tau-\tau|\le\frac{2\alpha\beta}{\beta+1} \sqrt{\frac{\log p}{n}}\right\}$. On $C$ holds $|\hat\tau^\ast-\tau|=|\hat\tau-\tau|\le\frac{2\alpha\beta}{\beta+1} \sqrt{\frac{\log p}{n}}\le \beta \min\left(|\tau|,\alpha\sqrt{\frac{\log p}{n}}\right)$. So $C\subset B$. Finally inequality (\ref{eqn:Hoeine}) yields 
\begin{align*}
%\P\left(\left|\tilde{\Sigma}_{ij}-\Sigma_{ij}\right|\le 4\min\left\{\sqrt{\Sigma_{ii}^{\ast}\Sigma_{jj}^{\ast}},~\gamma\sqrt{\frac{\log p}{n}}\right\}\right)\\
\P(B)&\ge\P\left(\left(\left\{\left|\hat\tau-\tau\right|\le \frac{2\alpha\beta}{\beta+1} \sqrt{\frac{\log p}{n}}\right\}\cap A^c\right) \cup A\right)\\
&\ge\P\left(\left|\hat{\tau}-\tau\right|\le \frac{2\alpha\beta}{\beta+1} \sqrt{\frac{\log p}{n}}\right)\\
&\ge 1-2p^{-\frac{\alpha^2\beta^2}{(\beta+1)^2}}.
\end{align*}
\item[(iii)] $|\tau|>\frac{2\alpha\beta}{\beta+1}  \sqrt{\frac{\log p}{n}}$: The union bound, the triangle inequality and inequality (\ref{eqn:Hoeine}) yield
{\allowdisplaybreaks
\begin{align*}
\P(B)&=\P\left(\left|\hat\tau^\ast-\tau\right|\le \alpha\beta\sqrt{\frac{\log p}{n}}\right)\ge \P\left(\left\{\left|\hat\tau-\tau\right|\le \alpha\beta\sqrt{\frac{\log p}{n}}\right\}\cap A^c\right)\\
&\ge 1-\P\left(\left|\hat\tau-\tau\right|> \alpha\beta\sqrt{\frac{\log p}{n}}\right)-\P(A)\\
&\ge 1-\P\left(\left|\hat\tau-\tau\right|> \alpha\beta\sqrt{\frac{\log p}{n}}\right)-\P\left(\left|\hat\tau-\tau\right|> \frac{\beta-1}{\beta+1}\alpha\sqrt{\frac{\log p}{n}}\right)\\
&\ge 1-2p^{-\frac{\alpha^2\beta^2}{4}}-2p^{-\frac{(\beta-1)^2\alpha^2}{4(\beta+1)^2}}\ge1-4p^{-\frac{(\beta-1)^2\alpha^2}{4(\beta+1)^2}}
\end{align*}}
\end{itemize}

We have in each case
\begin{align*}
\P(B)\ge 1-4p^{-\frac{(\beta-1)^2\alpha^2}{4(\beta+1)^2}}.
\end{align*}

Finally note that by the choice of $\alpha$ and $\beta$ 
\begin{align*}
\frac{(\beta-1)^2\alpha^2}{4(\beta+1)^2}>1.
\end{align*}
\end{proof}
\begin{proof}[Proof of Theorem  \ref{the:minuppbouken}]
\noindent The essential part of the proof of \eqref{eqn:minuppbouken} is to show the inequality 
\begin{align}\label{eqn:hilf1}
\underset{\tau\in \G_{\text{w},q}(c_{n,p})}{\sup}\frac{1}{p}\E\|\hat\tau^\ast-\tau\|_F^2\le C_{\alpha}c_{n,p}\left(\frac{\log p}{n}\right)^{1-\frac{q}{2}}.
\end{align}
Let $A_{ij}$ be the event that $\hat\tau^\ast_{ij}$ estimates  $\tau_{ij}$ by $0$, i.e. 
\begin{align*}
A_{ij}:=\left\{|\hat\tau_{ij}|\le \alpha\sqrt{\frac{\log p}{n}}\right\}.
\end{align*}
Moreover, define the event
\begin{align*}
B_{ij}:=\left\{ |\hat\tau^\ast_{ij}-\tau_{ij}|\le \beta \min\left(|\tau_{ij}|, \alpha\sqrt{\frac{\log p}{n}} \right) \right\}.
\end{align*}
We only prove the case $0<q<1$. The case $q=0$ is even easier and therefore omitted. Denote $[\cdot]_{i}:=[\cdot]_{\tau_i}$ and
\begin{align*}
(j)_{i}:=\begin{cases}[j]_{\tau_i}& \text{for }[j]_{\tau_i}<i,\\
[j]_{\tau_i}+1& \text{else.}
\end{cases}
\end{align*}
Analogously to \cite{Caib} we split $\frac{1}{p}\E \|\hat\tau^\ast-\tau\|_F^2$ into two parts
\begin{align*}
\frac{1}{p}\E \|\hat\tau^\ast&-\tau\|_F^2\\
%=~&\frac{1}{p}\sum_{i=1}^p\sum_{\substack{j=1\\j\neq i}}^p\E|\tau_{ij}-\hat\tau_{ij}|^2\\
=~&\frac{1}{p}\sum_{i=1}^p\sum_{\substack{j=1\\j\neq i}}^p\E|\tau_{i(j)_{i}}-\hat\tau_{i(j)_{i}}|^2\ind_{B_{i(j)_{i}}}+\frac{1}{p}\sum_{i=1}^p\sum_{\substack{j=1\\j\neq i}}^p\E|\tau_{i(j)_{i}}-\hat\tau_{i(j)_{i}}|^2\ind_{B^c_{i(j)_{i}}}\\
=:~&I_1+I_2.
\end{align*}
{\it Bounding} $I_1:$ First consider the case that 
\begin{align*}
\frac{(2-q)^{q/2}c_{n,p}n^{q/2}}{q^{q/2}\alpha^{q}(\log p)^{q/2}}<1.
\end{align*}
Then,
\begin{align}
I_1&=\frac{1}{p}\sum_{i=1}^p\sum_{\substack{j=1\\j\neq i}}^p\E|\tau_{i(j)_{i}}-\hat\tau_{i(j)_{i}}|^2\ind_{B_{i(j)_{i}}}\le \beta^2 c_{n,p}^{2/q}\left(1+\sum_{i=2}^p i^{-2/q}\right) \notag\\
&\le \beta^2 c_{n,p}^{2/q}\left(1+\int\limits_{1}^\infty x^{-2/q}dx\right)= \left(1+\frac{q}{2-q}\right)\beta^2 c_{n,p}c_{n,p}^{2/q-1}\notag\\
&\le 2\beta^2\alpha^{2-q}c_{n,p}\left(\frac{2-q}{q}\right)^{q/2}\left(\frac{\log p}{n}\right)^{1-q/2}.\label{eqn:one0}
\end{align}
Otherwise fix $m>1$. Then we have
{\allowdisplaybreaks \begin{align*}
I_1&=\frac{1}{p}\sum_{i=1}^p\sum_{\substack{j=1\\j\neq i}}^m\E|\tau_{i(j)_{i}}-\hat\tau_{i(j)_{i}}|^2\ind_{B_{i(j)_{i}}}+\frac{1}{p}\sum_{i=1}^p\sum_{\substack{j=m+1\\j\neq i}}^p\E|\tau_{i(j)_{i}}-\hat\tau_{i(j)_{i}}|^2\ind_{B_{i(j)_{i}}}\\
&\le m\beta^2\alpha^2\frac{\log p}{n}+\beta^2c_{n,p}^{2/q}\sum_{i=m+1}^p i^{-2/q}\\
&\le m\beta^2\alpha^2\frac{\log p}{n}+\beta^2c_{n,p}^{2/q}\int_m^\infty x^{-2/q}dx\\
&=m\beta^2\alpha^2\frac{\log p}{n}+\left(\frac{2}{q}-1\right)\beta^2c_{n,p}^{2/q}m^{-2/q+1}.
\end{align*}}
For $m=\left[\frac{(2-q)^{q/2}c_{n,p}n^{q/2}}{q^{q/2}\alpha^{q}(\log p)^{q/2}}\right]+1$ we get
\begin{align}
I_1\le \left(2\alpha^{-q}\left(\frac{2-q}{q}\right)^{q/2}+1\right)\beta^2\alpha^2c_{n,p}\left(\frac{\log p}{n}\right)^{1-q/2} \label{eqn:one1}.
\end{align}
{\it Bounding} $I_2:$ It remains to show that $I_2$ is of equal or smaller order than $I_1$. We have
\begin{align}
I_2&=p^{-1}\sum_{\substack{i,j=1\\i\neq j}}^p\E \left(\hat\tau_{ij}-\tau_{ij}\right)^2\ind_{B^c_{ij}\cap A^c_{ij}}+p^{-1}\sum_{\substack{i,j=1\\ i\neq j}}^p\E \left(\tau_{ij}^2\ind_{B^c_{ij}\cap A_{ij}}\right).\label{eqn:secter}
\end{align}
The first summand in (\ref{eqn:secter}) can be assessed by H\"older's inequality. For any $N\ge 3$ we have
\begin{align*}
p^{-1}\sum_{\substack{i,j=1\\i\neq j}}^p\E \left(\hat\tau_{ij}-\tau_{ij}\right)^2\ind_{B^c_{ij}\cap A^c_{ij}}&\le p^{-1}\sum_{\substack{i,j=1\\i \neq j}}^p \E^{1/N} \left(\hat\tau_{ij}-\tau_{ij}\right)^{2N}\P^{1-1/N}(B^c_{ij})\\
&\le \frac{16Np}{n}p^{-(1-1/N)(1+\varepsilon)},
\end{align*}
where we used inequality (\ref{eqn:Hoeine}) and Stirling's approximation to bound the expectation $\E \left(\hat\tau_{ij}-\tau_{ij}\right)^{2N}$ by the formula
\begin{align*}
\E \left(\hat\tau_{ij}-\tau_{ij}\right)^{2N}=\int\limits_{0}^\infty \P(\left(\hat\tau_{ij}-\tau_{ij}\right)^{2N}\ge x)dx
\end{align*}
and Lemma \ref{lem:lareve} for $\P^{1-1/N}(B^c_{ij})$.
By taking $N=\frac{1+\varepsilon}{\varepsilon}$ we conclude
\begin{align}
p^{-1}\sum_{\substack{i,j=1\\ i\neq j}}^p\E\left(\hat\tau_{ij}-\tau_{ij}\right)^2\ind_{B^c_{ij}\cap A^c_{ij}}&\le  16\frac{1+\varepsilon}{n\varepsilon}\notag
\\ &\le 16\frac{1+\varepsilon}{\varepsilon}c_{n,p}\left(\frac{\log p}{n}\right)^{1-q/2}.
\label{eqn:twoone1}
\end{align}
Lastly consider the second summand in (\ref{eqn:secter}). We observe that the event $B^c_{ij}\cap A_{ij}$ can only occur if $|\tau_{ij}|\ge \alpha\beta\sqrt{\frac{\log p}{n}}$. Therefore we obtain
\begin{align}
p^{-1}\sum_{\substack{i,j=1\\ i\neq j}}^p\E &\tau_{ij}^2\ind_{B^c_{ij}\cap A_{ij}}\\
&\hspace{-1cm}\le p^{-1}\sum_{\substack{i,j=1\\i\neq j}}^p\tau_{ij}^2\E\ind_{A_{ij}}\ind_{\left\{|\tau_{ij}|\ge \alpha\beta\sqrt{\frac{\log p}{n}}\right\}}\\
&\hspace{-1cm}\le~p^{-1}\sum_{\substack{i,j=1\\i\neq j}}^p\tau_{ij}^2\E\ind_{\left\{|\tau_{ij}|-|\hat{\tau}_{ij}-\tau_{ij}|< \alpha\sqrt{\frac{\log p}{n}}\right\}}\ind_{\left\{|\tau_{ij}|\ge \alpha\beta\sqrt{\frac{\log p}{n}}\right\}}\\
&\hspace{-1cm}\le~p^{-1}\sum_{\substack{i,j=1\\i \neq j}}^p \tau_{ij}^2 \E\ind_{\left\{|\hat{\tau}_{ij}-\tau_{ij}|>\left(1-\frac{1}{\beta}\right)|\tau_{ij}|\right\}}\ind_{\left\{|\tau_{ij}|\ge \alpha\beta\sqrt{\frac{\log p}{n}}\right\}}\\
&\hspace{-1cm}\le~p^{-1} \frac{2}{n}\sum_{i,j=1}^p n\tau_{ij}^2 \exp\left(-\frac{\left(1-\beta^{-1}\right)^2n\tau_{ij}^2}{4}\right)\ind_{\left\{|\tau_{ij}|\ge \alpha\beta\sqrt{\frac{\log p}{n}}\right\}}\\
&\hspace{-1cm}\le~p^{-2}\frac{2\alpha^2\beta^2}{3n}\sum_{i,j=1}^p \frac{3n\tau_{ij}^2}{\alpha^2\beta^2}  \exp\left(-\frac{3n\tau_{ij}^2}{\alpha^2\beta^2}\right)\label{eqn:hilf2}\\
&\hspace{-1cm}\le~\frac{2\alpha^2\beta^2}{3en}.\label{eqn:twotwo1}
\end{align}
In line \eqref{eqn:hilf2} we have splitted the exponential term into
\begin{align}\label{eqn:betbel}
\exp\left(-\frac{(1-\beta^{-1})^2n\tau_{ij}}{16}\right)~~\text{and}~~\exp\left(-\frac{3(1-\beta^{-1})^2n\tau_{ij}}{16}\right),
\end{align}
where the first term is bounded by $p^{-1}$ and the second one by
$$
 \exp\left(-\frac{3n\tau_{ij}^2}{\alpha^2\beta^2}\right).
$$
The last line \eqref{eqn:twotwo1} follows by the the fact the function $f(x)=x\exp(-x)$ is bounded from above by $e^{-1}$. 
%Therefore,
%\begin{align}
%p^{-1}\sum_{\substack{i,j=1\\ i\neq j}}^p\E \tau_{ij}^2\ind_{B^c_{ij}\cap A_{ij}}\le \frac{2\alpha^2\beta^2}{3e}c_{n,p}\left(\frac{\log p}{n}\right)^{q/2}\label{eqn:twotwo1}.
%\end{align}
Summarizing the bounds (\ref{eqn:one0}), (\ref{eqn:one1}), (\ref{eqn:twoone1}) and (\ref{eqn:twotwo1}) yields inequality \eqref{eqn:hilf1}. Now inequality \eqref{eqn:minuppbouken} is an easy conclusion of  \eqref{eqn:hilf1} since
\begin{align*}
\underset{\tau\in \G_{\text{w},q}(c_{n,p})}{\sup}\E\|\hat\tau^\ast-\tau\|_F^2\ge \frac{4}{\pi^2}\underset{\rho\in \G_{\text{w},q}(c_{n,p})}{\sup}\E\|\hat\rho^\ast-\rho\|_F^2.
\end{align*}
Finally, to derive inequality \eqref{eqn:minuppboucor} note that
\begin{align*}
&\underset{\rho\in \G_{\text{w},q}(c_{n,p})}{\sup}\E\|\hat\rho^\ast-\rho\|_2^2\\
&\hspace{1.7cm}\le \underset{\rho\in \G_{\text{w},q}(c_{n,p})}{\sup}\E\|\hat\rho^{\ast,B}-\rho^B\|_1^2+\underset{\rho\in \G_{\text{w},q}(c_{n,p})}{\sup}\E\|\hat\rho^{\ast,B^c}-\rho^{B^c}\|_F^2\\
&\hspace{1.7cm}\le \frac{\pi^2}{4}\left(\underset{\tau\in \G_{\text{w},q}(c_{n,p})}{\sup}\E\|\hat\tau^{\ast,B}-\tau^B\|_1^2+\underset{\tau\in \G_{\text{w},q}(c_{n,p})}{\sup}\E\|\hat\tau^{\ast,B^c}-\tau^{B^c}\|_F^2\right),
\end{align*}
where 
\begin{align*}
\hat\rho^{\ast,B}_{ij}-\rho^B_{ij}:=\left(\hat\rho_{ij}-\rho_{ij}\right)\ind_{B_{ij}}&,~~\hat\rho^{\ast,B^c}_{ij}-\rho^{B^c}_{ij}:=\left(\hat\rho_{ij}-\rho_{ij}\right)\ind_{B_{ij}^c},\\
\hat\rho^{\ast,B}_{ij}-\rho^B_{ij}:=\left(\hat\tau_{ij}-\tau_{ij}\right)\ind_{B_{ij}}&,~~\hat\tau^{\ast,B^c}_{ij}-\tau^{B^c}_{ij}:=\left(\hat\tau_{ij}-\tau_{ij}\right)\ind_{B^c_{ij}}.
\end{align*}
The expression 
$$\underset{\tau\in \G_{\text{w},q}(c_{n,p})}{\sup}\E\|\hat\tau^{\ast,B}-\tau^B\|_1^2$$
can be bounded by
$C_{\alpha,\beta,q}c_{n,p}^2\left(\frac{\log p}{n}\right)^{1-q}$ for some constant $C_{\alpha,\beta,q}>0$
using the same arguments as in (51) of \cite{Caib}. Moreover, it holds 
\begin{align*}
\underset{\tau\in \G_{\text{w},q}(c_{n,p})}{\sup}\E\|\hat\tau^{\ast,B^c}-\tau^{B^c}\|_F^2=pI_2.
\end{align*}
Note that in the calculation of $I_2$ for $\alpha>2\sqrt{2}$ we have $\varepsilon>1$, and the first expression in \eqref{eqn:betbel} can be made smaller than any power of $1/p$ by choosing $\beta$ sufficiently large. This proves the claim.
\end{proof}
\begin{proof}[Proof of Theorem \ref{the:minuppboucor2}]
By definition an entry $\sin[\frac{\pi}{2}\hat\tau]_{ij},~i\neq j$ will be rejected iff $|\sin[\frac{\pi}{2}\hat\tau]_{ij}|\le \alpha\sqrt\frac{\log p}{n}$. Rearranging yields
\begin{align*}
|\sin[\frac{\pi}{2}\hat\tau]_{ij}|\le \alpha\sqrt\frac{\log p}{n}&\Longleftrightarrow |\sin[\frac{\pi}{2}\hat\tau]_{ij}|\le \alpha\sqrt\frac{\log p}{n}\wedge 1\\
&\Longleftrightarrow |\hat\tau_{ij}|\le \frac{2}{\pi}\arcsin\left(\alpha\sqrt\frac{\log p}{n}\wedge 1\right).
\end{align*}

By the mean value theorem there exists $\theta\in\left(0,\alpha\sqrt\frac{\log p}{n}\wedge1\right)$ such that $$\arcsin\left(\alpha\sqrt\frac{\log p}{n}\wedge 1\right)=\left(\alpha\sqrt\frac{\log p}{n}\wedge 1\right)\frac{1}{\sqrt{1-\theta^2}}.$$

Moreover by convexity of the arcsine function on $[0,1]$ we have $\lambda:=\frac{1}{\sqrt{1-\theta^2}}\in [1,\frac{\pi}{2}]$, where $\lambda=\frac{\pi}{2}$ iff $\alpha\sqrt\frac{\log p}{n}\ge 1$. Therefore we have $\hat\rho=T_\alpha(\sin[\frac{\pi}{2}\hat\tau])=\sin[\frac{\pi}{2}T_{\frac{2}{\pi}\alpha\lambda}(\hat\tau)]$. Finally we conclude by Theorem \ref{the:minuppboucor} that
\begin{align*}
\underset{\rho\in \G_{\text{w},q}(c_{n,p})}{\sup}\frac{1}{p}\E\|\hat\rho-\rho\|_F^2&\le \underset{\lambda\in [1,\frac{\pi}{2}]}{\sup}\underset{\rho\in \G_{\text{w},q}(c_{n,p})}{\sup}\frac{1}{p}\E\|\sin[\frac{\pi}{2}T_{\frac{2}{\pi}\alpha\lambda}(\hat\tau)]-\rho\|_F^2\\
&= \underset{\tilde\alpha\in [\frac{2}{\pi}\alpha,\alpha]}{\sup}~\underset{\rho\in \G_{\text{w},q}(c_{n,p})}{\sup}\frac{1}{p}\E\|\sin[\frac{\pi}{2}T_{\frac{2}{\pi}\alpha\lambda}(\hat\tau)]-\rho\|_F^2\\
&\le \underset{\tilde\alpha\in [\frac{2}{\pi}\alpha,\alpha]}{\sup} C_{\tilde\alpha}c_{n,p}\left(\frac{\log p}{n}\right)^{1-\frac{q}{2}}\\
&\le C_{\alpha}c_{n,p}\left(\frac{\log p}{n}\right)^{1-\frac{q}{2}},
\end{align*}
where by the proof of Theorem \ref{the:minuppboucor} the expression $\underset{\tilde\alpha\in [\frac{2}{\pi}\alpha,\alpha]}{\sup} C_{\tilde\alpha}$ is bounded for any $\alpha>\pi$.
\end{proof}

Next we provide the proof of Theorem \ref{the:optuppbou}. It is sufficient to restrict to the case that $\alpha\le 2$ in the upper bound (\ref{eqn:optuppboucor}). In view of the proof of Theorem \ref{the:minuppbouken} we need to improve Lemma $\ref{lem:lareve}$. This can be done by distinguishing between entries $\hat\tau_{ij}^\ast$ based on weakly correlated components and all the other entries $\hat\tau_{ij}^\ast$. Clearly, if $\tau_{ij}$ is sufficiently large compared to the threshold level $\alpha(\frac{\log p}{n})^{1/2}$, then $|\hat\tau_{ij}^\ast-\tau_{ij}|^2=|\hat\tau_{ij}-\tau_{ij}|^2=O(\frac{\log p}{n})$ with an appropriately large probability. The next lemma gives a more precise statement of the last idea. 

\begin{lemma}\label{lem:constrcor}
Let $(Y_1,Z_1),...,(Y_n,Z_n)\sim \mathcal{N}_2(\mu,\Sigma)$ with correlation $|\rho|\ge \frac{5\pi}{2}\sqrt{\frac{\log p}{n}}$. Then, for any $\frac{2\sqrt{2}}{3}<\alpha\le 2$ we have
\begin{align*}
\P\left(|\hat\tau^\ast(Y_{1:n},Z_{1:n})-\tau(Y_1,Z_1)|\le 3\sqrt{\frac{\log p}{n}}\right)\ge 1-4p^{-\frac{9}{4}}.
\end{align*}
\end{lemma}
\begin{proof}
The proof is similar to the third part of Lemma \ref{lem:lareve}. Let $\tau:=\tau(Y_1,Z_1)$, $\hat\tau:=\hat\tau(Y_{1:n},Z_{1:n})$ and $\hat\tau^\ast:=\hat\tau^\ast(Y_{1:n},Z_{1:n})$. First note that $|\tau|\ge \frac{2}{\pi}|\rho|\ge 5\sqrt\frac{\log p}{n}$.
Recall that 
\begin{align*}
A=\left\{|\hat\tau|< \alpha \sqrt{\frac{\log p}{n}}\right\}.
\end{align*}

Then we have by Hoeffding's inequality
\begin{align*}
\P\left(|\hat\tau^\ast-\tau|\ge 3\sqrt{\frac{\log p}{n}}\right)&\ge \P\left(\left\{|\hat\tau-\tau|\ge 3\sqrt{\frac{\log p}{n}}\right\}\cap A^c\right)\\
&\ge 1-2p^{-\frac{9}{4}}-\P\left(|\hat\tau|< \alpha \sqrt{\frac{\log p}{n}}\right)\\
&\ge 1-2p^{-\frac{9}{4}}-\P\left(|\hat\tau-\tau|\ge 3\sqrt{\frac{\log p}{n}}\right)\\
&\ge1-4p^{-\frac{9}{4}},
\end{align*}
where the second last line follows by triangle inequality since $\alpha\le 2$, $|\hat\tau|<\alpha\sqrt\frac{\log p}{n}$ and $|\tau|\ge 5\sqrt\frac{\log p}{n}$ implies that $|\hat\tau-\tau|\ge 3\sqrt\frac{\log p}{n}$. Hence the claim holds true.
\end{proof} 

Now we can give a more refined version of Lemma $\ref{lem:lareve}$ for Gaussian random vectors by treating the entries $\hat\tau_{ij}^\ast$ based on weakly correlated components more carefully.
\begin{lemma}\label{lem:concor}
Let $(Y_1,Z_1),...,(Y_n,Z_n)\widesim{i.i.d.} \mathcal{N}_2(\mu,\Sigma)$, where $p\le n^{\eta_u}$. Then for any $\frac{2\sqrt{2}}{3}<\alpha\le 2$ and $\beta>\frac{3\alpha+2\sqrt{2}}{3\alpha-2\sqrt{2}}$ let the event 
$$B:=\left\{|\hat\tau^\ast(Y_{1:n},Z_{1:n})-\tau(Y_1,Z_1)|\le\beta\min\left(|\tau(Y_1,Z_1)|,\alpha\sqrt{\frac{\log p}{n}}\right)\right\}.$$
Then we have 
\begin{align*}
\P(B)\ge 1-C_{\alpha,\beta,\eta_u}p^{-(1+\varepsilon)},
\end{align*}
where $\varepsilon>0$ depends on $\alpha$ and $\beta$, and $\varepsilon$ be chosen larger than $1$ for $\alpha>4/3$.
\end{lemma}
\begin{proof}
For $|\rho(Y_1,Z_1)|\ge \frac{5\pi}{2}\sqrt\frac{\log p}{n}$ we have $$\beta\min\left(|\tau(Y_1,Z_1)|,\alpha\sqrt{\frac{\log p}{n}}\right)\ge\beta\alpha\sqrt\frac{\log p}{n}>3\sqrt\frac{\log p}{n}$$ and therefore by Lemma \ref{lem:constrcor} the inequality holds here. Otherwise for $|\rho(Y_1,Z_1)|$ $<\frac{5\pi}{2}\sqrt\frac{\log p}{n}$  we just need to replace Hoeffding's inequality in the proof of Lemma \ref{lem:lareve} by the upper bound on the probability that $\hat\tau(Y_{1:n},Z_{1:n})$ is close to its mean $\tau(Y_1,Z_1)$  in Corollary \ref{lem:conweacor}.
\end{proof} 
\begin{proof}[Proof of Theorem \ref{the:optuppbou}]
The proof of inequality (\ref{eqn:optuppboucor}) is essentially analogous to the one of Theorem \ref{the:minuppbouken}. Using Lemma \ref{lem:concor} instead of Lemma \ref{lem:lareve} to compute $\P^{1-1/N}(B_{ij}^c)$ provides the desired upper bound. We choose an appropriate example to show that for $\alpha<\frac{2\sqrt{2}}{3}$ the corresponding estimator $\hat\rho^\ast$ in general does not attain the minimax rate. Since $\Id\in\mathcal{G}_{\text{w},q}(c_{n,p})$ for all $0\le q<1$ and all $c_{n,p}>0$, we assume that $\Sigma=\Id$. Furthermore let $\varepsilon:=1-\frac{9}{8}\alpha^2$ and $c_{n,p}=o\left(\left(\frac{\log p}{n}\right)^{q/2}\frac{p^\varepsilon}{\sqrt{\log p}}\right)$. The lower bound in Corollary \ref{lem:conweacor} provides 
\begin{align*}
\E p^{-1}\|\hat\tau^\ast-\tau\|_F^2&\ge p^{-1}\sum_{i,j=1}^p\left(\hat\tau^\ast_{ij}-\tau_{ij}\right)^2\ind_{A_{ij}^c}\\
&\ge \frac{\alpha^2\log p}{np}\sum_{\substack{i,j=1\\i\neq j}}^p\P(A_{ij}^c)\\
&\ge C_{\alpha,\eta_u}\frac{p^\varepsilon\log p}{n\sqrt{\log p}}\\
&= C_{\alpha,\eta_u}c_{n,p}\left(\frac{\log p}{n}\right)^{1-\frac{q}{2}}\frac{\frac{p^\varepsilon}{\sqrt{\log p}}\left(\frac{\log p}{n}\right)^{\frac{q}{2}}}{c_{n,p}},
\end{align*}
where we used the well-known inequality
\begin{align*}
\frac{\phi(x)}{x+1/x}<1-\Phi(x).
\end{align*}
%The last line proves that $\hat\tau^\ast$ does not attain the minimax rate. 
%It remains to give a suitable lower bound on $\sup_{\rho\in\G_{\text{w},q}(c_{n,p})}\E\|\sin[\frac{\pi}{2}\hat\tau^\ast]-\rho\|_F^2$. Again we first bound this expression by the special case $\Sigma=\Id$. We have
Hence, we obtain
\begin{align*}
\sup_{\rho\in\G_{\text{w},q}(c_{n,p})}p^{-1}\E\|\sin[\frac{\pi}{2}\hat\tau^\ast]-\rho\|_F^2&\ge p^{-1}\E\|\sin[\frac{\pi}{2}\hat\tau^\ast]-\Id\|_F^2\\
&\ge  p^{-1}\E\|\hat\tau^\ast-\Id\|_F^2\\
&\ge C_{\alpha,\varepsilon,\eta_u}c_{n,p}\left(\frac{\log p}{n}\right)^{1-\frac{q}{2}}\frac{\frac{p^\varepsilon}{\sqrt{\log p}}\left(\frac{\log p}{n}\right)^{\frac{q}{2}}}{c_{n,p}}.
\end{align*}
The last line proves that $\hat\rho^\ast$ does not attain the minimax rate. The last statement of the theorem follows analogously, where the upper bound on $c_{n,p}$ in Assumption $(A_3)$ is to be taken into consideration.
\end{proof}
\begin{proof}[Proof of Theorem \ref{the:optuppbou2}]
Inequality (\ref{eqn:optuppbou2}) follows similarly to inequality (\ref{eqn:minuppboucor2}). For the proof of the lower bound fix $\alpha<\sqrt{2}\pi/3$ let the sample size $n$ be sufficiently large such that for some $\alpha<\tilde\alpha<\sqrt{2}\pi/3$ the implication  
\begin{align*}
|\sin(\frac{\pi}{2}x)|\le \alpha\sqrt\frac{\log p}{n}\Longrightarrow |x|\le \frac{2}{\pi}\tilde\alpha\sqrt\frac{\log p}{n}
\end{align*}
holds true. Then, we conclude
\begin{align*}
\sup_{\rho\in\G_{\text{w},q}(c_{n,p})}p^{-1}\E\|\hat\rho-\rho\|_F^2&\ge p^{-1}\E\|T_\alpha(\sin[\frac{\pi}{2}\hat\tau])-\Id\|_F^2\\
&\ge p^{-1}\E\|\sin[\frac{\pi}{2}\hat\tau^\ast_{\frac{2}{\pi}\tilde\alpha}]-\Id\|_F^2\\
&\ge C_{\frac{2}{\pi}\tilde\alpha,\varepsilon,\eta_u}c_{n,p}\left(\frac{\log p}{n}\right)^{1-\frac{q}{2}}\frac{p^\varepsilon\left(\frac{\log p}{n}\right)^{\frac{q}{2}}}{c_{n,p}},
\end{align*}
where $\varepsilon=\frac{9}{8}-\left(\frac{2\tilde\alpha}{\pi}\right)^2$. Thus, the minimax rate is not attained for $$c_{n,p}=o\left(\left(\frac{\log p}{n}\right)^{q/2}\frac{p^\varepsilon}{\sqrt{\log p}}\right).$$
Analogously we obtain the last statement of the theorem.
\end{proof}
\begin{proof}[Proof of Theorem \ref{optspe}]
The first part of the result follows analogously to \eqref{eqn:minuppboucor}. We obtain the second part of the theorem by the elementary inequality
\begin{align*}
\sup_{\rho\in\G_{\text{w},q}(c_{n,p})}\E\|\sin[&\frac{\pi}{2}\hat\tau^\ast]-\rho\|_2^2\ge \sup_{\rho\in\G_{\text{w},q}(c_{n,p})}\frac{1}{p}\E\|\sin[\frac{\pi}{2}\hat\tau^\ast]-\rho\|_F^2.
%&\ge C_{\alpha,}
\end{align*}
and the bound
\begin{align*}
\sup_{\rho\in\G_{\text{w},0}(c_{n,p})}\frac{1}{p}\E\|\sin[\frac{\pi}{2}\hat\tau^\ast]-\rho\|_F^2\ge C_{\alpha,\eta_u}c_{n,p}^2\frac{\log p}{n}\frac{p^\varepsilon}{c_{n,p}^2\sqrt{\log p}}
\end{align*}
from the proof of Theorem \ref{the:optuppbou}, where we assume that $$c_{n,p}=o\left(\frac{p^{\varepsilon/2}}{(\log p)^{1/4}}\right).$$

\end{proof}

\subsection*{Acknowledgments}The present work is part of my Ph.D. thesis. I am grateful to my supervisor Angelika Rohde for her constant encouragement and support.
I sincerely thank Fang Han for pointing out a mistake in an earlier version of this paper. I also thank Holger Dette and Tim Patschkowski for helpful suggestions to improve the presentation of the paper.\newline
This work was supported by the {\it Deutsche Forschungsgemeinschaft} research unit 1735, Ro 3766/3-1.

\bibliographystyle{imsart-nameyear}
\bibliography{reference}

\begin{thebibliography}{42}
% BibTex style file: imsart-nameyear.bst, 2013-01-28
% Default style options (sort=1,type=nameyear).
% Used options (sort=1,type=nameyear).

\bibitem[\protect\citeauthoryear{Abramovich et~al.}{2006}]{Abramovich2006}
\begin{barticle}[author]
\bauthor{\bsnm{Abramovich},~\bfnm{F.}\binits{F.}},
  \bauthor{\bsnm{Benjamini},~\bfnm{Y.}\binits{Y.}},
  \bauthor{\bsnm{Donoho},~\bfnm{D.}\binits{D.}} \AND
  \bauthor{\bsnm{Johnstone},~\bfnm{I.}\binits{I.}}
(\byear{2006}).
\btitle{Adapting to unknown sparsity by controlling the false discovery rate}.
\bjournal{Ann. Stat.}
\bvolume{34}
\bpages{584-653}.
\end{barticle}
\endbibitem

\bibitem[\protect\citeauthoryear{Amini and Wainwright}{2009}]{Amini2009}
\begin{barticle}[author]
\bauthor{\bsnm{Amini},~\bfnm{A.}\binits{A.}} \AND
  \bauthor{\bsnm{Wainwright},~\bfnm{M.}\binits{M.}}
(\byear{2009}).
\btitle{High-dimensional analysis of semidefinite relaxations for sparse
  principal components}.
\bjournal{Ann. Statist.}
\bvolume{37}
\bpages{2877-2921}.
\end{barticle}
\endbibitem

\bibitem[\protect\citeauthoryear{Baik and Silverstein}{2006}]{Baik2006}
\begin{barticle}[author]
\bauthor{\bsnm{Baik},~\bfnm{Jinho}\binits{J.}} \AND
  \bauthor{\bsnm{Silverstein},~\bfnm{Jack~W.}\binits{J.~W.}}
(\byear{2006}).
\btitle{Eigenvalues of large sample covariance matrices of spiked population
  models}.
\bjournal{J. Multivariate Anal.}
\bvolume{97}.
\end{barticle}
\endbibitem

\bibitem[\protect\citeauthoryear{Berthet and Rigollet}{2013}]{Berthet2013}
\begin{barticle}[author]
\bauthor{\bsnm{Berthet},~\bfnm{Q.}\binits{Q.}} \AND
  \bauthor{\bsnm{Rigollet},~\bfnm{P.}\binits{P.}}
(\byear{2013}).
\btitle{Optimal detection of sparse principal components in high dimensions}.
\bjournal{Ann. Statist.}
\bvolume{41}
\bpages{1780-1815}.
\end{barticle}
\endbibitem

\bibitem[\protect\citeauthoryear{Bickel and Levina}{2008a}]{Bickel2008b}
\begin{barticle}[author]
\bauthor{\bsnm{Bickel},~\bfnm{P.}\binits{P.}} \AND
  \bauthor{\bsnm{Levina},~\bfnm{E.}\binits{E.}}
(\byear{2008}a).
\btitle{Covariance regularization by thresholding}.
\bjournal{Ann. Stat.}
\bvolume{36}
\bpages{2577-2604}.
\end{barticle}
\endbibitem

\bibitem[\protect\citeauthoryear{Bickel and Levina}{2008b}]{Bickel2008}
\begin{barticle}[author]
\bauthor{\bsnm{Bickel},~\bfnm{P.}\binits{P.}} \AND
  \bauthor{\bsnm{Levina},~\bfnm{E.}\binits{E.}}
(\byear{2008}b).
\btitle{Regularized estimation of large covariance matrices}.
\bjournal{Ann. Stat.}
\bvolume{36}
\bpages{199-227}.
\end{barticle}
\endbibitem

\bibitem[\protect\citeauthoryear{Borsdorf and Higham}{2010}]{Borsdorf}
\begin{barticle}[author]
\bauthor{\bsnm{Borsdorf},~\bfnm{R.}\binits{R.}} \AND
  \bauthor{\bsnm{Higham},~\bfnm{N.}\binits{N.}}
(\byear{2010}).
\btitle{A preconditioned Newton algorithm for the nearest correlation matrix}.
\bjournal{IMA J. Numer. Anal.}
\bvolume{30}.
\end{barticle}
\endbibitem

\bibitem[\protect\citeauthoryear{Cai and Liu}{2011}]{Cai2011}
\begin{barticle}[author]
\bauthor{\bsnm{Cai},~\bfnm{T.}\binits{T.}} \AND
  \bauthor{\bsnm{Liu},~\bfnm{W.}\binits{W.}}
(\byear{2011}).
\btitle{Adaptive thresholding for sparse covariance matrix estimation}.
\bjournal{J. Amer. Statist. Assoc.}
\bvolume{106}.
\end{barticle}
\endbibitem

\bibitem[\protect\citeauthoryear{Cai, Ma and Wu}{2013}]{Cai2012}
\begin{barticle}[author]
\bauthor{\bsnm{Cai},~\bfnm{T.}\binits{T.}},
  \bauthor{\bsnm{Ma},~\bfnm{Z.}\binits{Z.}} \AND
  \bauthor{\bsnm{Wu},~\bfnm{Y.}\binits{Y.}}
(\byear{2013}).
\btitle{Sparse PCA: Optimal rates and adaptive estimation}.
\bjournal{Ann. Statist.}
\bvolume{41}
\bpages{3074-3110}.
\end{barticle}
\endbibitem

\bibitem[\protect\citeauthoryear{Cai, Ma and Wu}{2014+}]{Caid}
\begin{bunpublished}[author]
\bauthor{\bsnm{Cai},~\bfnm{T.}\binits{T.}},
  \bauthor{\bsnm{Ma},~\bfnm{Z.}\binits{Z.}} \AND
  \bauthor{\bsnm{Wu},~\bfnm{Y.}\binits{Y.}}
(\byear{2014}+).
\btitle{Optimal estimation and rank detection for sparse spiked covariance
  matrices}.
\bnote{To appear in: {\it Prob. Theory Rel. Fields}}.
\end{bunpublished}
\endbibitem

\bibitem[\protect\citeauthoryear{Cai and Zhou}{2012a}]{Caib}
\begin{barticle}[author]
\bauthor{\bsnm{Cai},~\bfnm{T.}\binits{T.}} \AND
  \bauthor{\bsnm{Zhou},~\bfnm{H.}\binits{H.}}
(\byear{2012}a).
\btitle{Optimal rates of convergence for sparse covariance matrix estimation}.
\bjournal{Ann. Stat.}
\bvolume{40}
\bpages{2389-2420}.
\end{barticle}
\endbibitem

\bibitem[\protect\citeauthoryear{Cai and Zhou}{2012b}]{Caic}
\begin{barticle}[author]
\bauthor{\bsnm{Cai},~\bfnm{T.}\binits{T.}} \AND
  \bauthor{\bsnm{Zhou},~\bfnm{H.}\binits{H.}}
(\byear{2012}b).
\btitle{Covariance matrix estimation under $l_1$ norm}.
\bjournal{Statist. Sinicia}
\bvolume{22}
\bpages{1319-1378}.
\end{barticle}
\endbibitem

\bibitem[\protect\citeauthoryear{Cand\`es et~al.}{2010}]{Candes2010}
\begin{barticle}[author]
\bauthor{\bsnm{Cand\`es},~\bfnm{E.}\binits{E.}},
  \bauthor{\bsnm{Eldar},~\bfnm{Y.}\binits{Y.}},
  \bauthor{\bsnm{Needell},~\bfnm{D.}\binits{D.}} \AND
  \bauthor{\bsnm{Randall},~\bfnm{P.}\binits{P.}}
(\byear{2010}).
\btitle{Compressed sensing with coherent and redundant dictionaries}.
\bjournal{Appl. Comput. Harmon. Anal.}
\bvolume{31}
\bpages{59-73}.
\end{barticle}
\endbibitem

\bibitem[\protect\citeauthoryear{Clark}{2000}]{Clark2000}
\begin{barticle}[author]
\bauthor{\bsnm{Clark},~\bfnm{L.}\binits{L.}}
(\byear{2000}).
\btitle{An Asymptotic Expansion for the Number of Permutations with a Certain
  Number of Inversions}.
\bjournal{Electron. J. Comb.}
\bvolume{7}.
\end{barticle}
\endbibitem

\bibitem[\protect\citeauthoryear{Cramer}{1750}]{Cramer1750}
\begin{bbook}[author]
\bauthor{\bsnm{Cramer},~\bfnm{G.}\binits{G.}}
(\byear{1750}).
\btitle{Introduction \`a l'analyse des lignes courbes alg\'ebriques}.
\bpublisher{Gen\`eve: Fr\`eres Cramer et C. Philibert}.
\end{bbook}
\endbibitem

\bibitem[\protect\citeauthoryear{D'Aspremont et~al.}{2007}]{daspremont2007}
\begin{binproceedings}[author]
\bauthor{\bsnm{D'Aspremont},~\bfnm{A.}\binits{A.}},
  \bauthor{\bsnm{El~Ghaoui},~\bfnm{L.}\binits{L.}},
  \bauthor{\bsnm{Jordan},~\bfnm{M.}\binits{M.}} \AND
  \bauthor{\bsnm{Lanckriet},~\bfnm{G.}\binits{G.}}
(\byear{2007}).
\btitle{A direct formulation for sparse PCA using semidefinite programming}.
In \bbooktitle{SIAM Rev.}
\bvolume{49}
\bpages{434-448}.
\end{binproceedings}
\endbibitem

\bibitem[\protect\citeauthoryear{Donoho}{2006}]{Donoho2006}
\begin{barticle}[author]
\bauthor{\bsnm{Donoho},~\bfnm{D.}\binits{D.}}
(\byear{2006}).
\btitle{Compressed sensing}.
\bjournal{IEEE Trans. Inform. Theory}
\bvolume{52}
\bpages{1289-1306}.
\end{barticle}
\endbibitem

\bibitem[\protect\citeauthoryear{El~Karoui}{2008}]{Karoui2008}
\begin{barticle}[author]
\bauthor{\bsnm{El~Karoui},~\bfnm{N.}\binits{N.}}
(\byear{2008}).
\btitle{Operator norm consistent estimation of large-dimensional sparse
  covariance matrices}.
\bjournal{Ann. Stat.}
\bvolume{36}
\bpages{2717-2756}.
\end{barticle}
\endbibitem

\bibitem[\protect\citeauthoryear{Fan, Yuan and Mincheva}{2013}]{Fan2013}
\begin{barticle}[author]
\bauthor{\bsnm{Fan},~\bfnm{J.}\binits{J.}},
  \bauthor{\bsnm{Yuan},~\bfnm{L.}\binits{L.}} \AND
  \bauthor{\bsnm{Mincheva},~\bfnm{M}\binits{M.}}
(\byear{2013}).
\btitle{Large covariance estimation by thresholding principal orthogonal
  complements}.
\bjournal{J. Roy. Statist. Soc. Ser. B}
\bvolume{75}
\bpages{603-680}.
\end{barticle}
\endbibitem

\bibitem[\protect\citeauthoryear{Fang, Fang and Kotz}{2002}]{Fang2002}
\begin{barticle}[author]
\bauthor{\bsnm{Fang},~\bfnm{H.}\binits{H.}},
  \bauthor{\bsnm{Fang},~\bfnm{K.}\binits{K.}} \AND
  \bauthor{\bsnm{Kotz},~\bfnm{S.}\binits{S.}}
(\byear{2002}).
\btitle{The meta-elliptical distributions with given marginals}.
\bjournal{J. Mult. Anal.}
\bvolume{82}.
\end{barticle}
\endbibitem

\bibitem[\protect\citeauthoryear{Fang, Kotz and Ng}{1990}]{Fang1990}
\begin{bbook}[author]
\bauthor{\bsnm{Fang},~\bfnm{K.}\binits{K.}},
  \bauthor{\bsnm{Kotz},~\bfnm{S.}\binits{S.}} \AND
  \bauthor{\bsnm{Ng},~\bfnm{K.}\binits{K.}}
(\byear{1990}).
\btitle{Symmetric Multivariate and Related Distributions}.
\bpublisher{Chapman \& Hall}.
\end{bbook}
\endbibitem

\bibitem[\protect\citeauthoryear{Han and Liu}{2013}]{Han2013}
\begin{bunpublished}[author]
\bauthor{\bsnm{Han},~\bfnm{F.}\binits{F.}} \AND
  \bauthor{\bsnm{Liu},~\bfnm{H.}\binits{H.}}
(\byear{2013}).
\btitle{Statistical analysis of latent generalized correlation matrix
  estimation in transellipitical distribution}.
\bnote{Preprint}.
\end{bunpublished}
\endbibitem

\bibitem[\protect\citeauthoryear{Han, Liu and Zhao}{2013}]{Han2013b}
\begin{barticle}[author]
\bauthor{\bsnm{Han},~\bfnm{F.}\binits{F.}},
  \bauthor{\bsnm{Liu},~\bfnm{H.}\binits{H.}} \AND
  \bauthor{\bsnm{Zhao},~\bfnm{T.}\binits{T.}}
(\byear{2013}).
\btitle{CODA: High dimensional copula discriminant analysis}.
\bjournal{J. Mach. Learn. Res.}
\bvolume{14}
\bpages{629-671}.
\end{barticle}
\endbibitem

\bibitem[\protect\citeauthoryear{Han and Liu}{2014}]{Han2014}
\begin{barticle}[author]
\bauthor{\bsnm{Han},~\bfnm{F.}\binits{F.}} \AND
  \bauthor{\bsnm{Liu},~\bfnm{H.}\binits{H.}}
(\byear{2014}).
\btitle{Scale-invariant sparse PCA on high dimensional meta-elliptical data}.
\bjournal{J. Amer. Statist. Assoc.}
\bvolume{109}
\bpages{275-287}.
\end{barticle}
\endbibitem

\bibitem[\protect\citeauthoryear{Hoeffding}{1963}]{Hoeffding1963}
\begin{barticle}[author]
\bauthor{\bsnm{Hoeffding},~\bfnm{W.}\binits{W.}}
(\byear{1963}).
\btitle{Probability inequalities for sums of bounded random variables}.
\bjournal{J. Amer. Statist. Assoc.}
\bvolume{58}
\bpages{13-30}.
\end{barticle}
\endbibitem

\bibitem[\protect\citeauthoryear{Huang and Shen}{2008}]{Shen2008}
\begin{barticle}[author]
\bauthor{\bsnm{Huang},~\bfnm{J.}\binits{J.}} \AND
  \bauthor{\bsnm{Shen},~\bfnm{H.}\binits{H.}}
(\byear{2008}).
\btitle{Sparse principal component analysis via regularized low rank matrix
  approximation}.
\bjournal{J. Multivariate Anal.}
\bvolume{99}
\bpages{1015-1034}.
\end{barticle}
\endbibitem

\bibitem[\protect\citeauthoryear{Hult and Lindskog}{2002}]{Hult2002}
\begin{barticle}[author]
\bauthor{\bsnm{Hult},~\bfnm{H.}\binits{H.}} \AND
  \bauthor{\bsnm{Lindskog},~\bfnm{F.}\binits{F.}}
(\byear{2002}).
\btitle{Multivariate extremes, aggregation and dependence in elliptical
  distributions}.
\bjournal{Adv. in Appl. Probab.}
\bvolume{34}
\bpages{587-608}.
\end{barticle}
\endbibitem

\bibitem[\protect\citeauthoryear{Johnstone}{2001}]{Johnstone2001}
\begin{barticle}[author]
\bauthor{\bsnm{Johnstone},~\bfnm{Iain}\binits{I.}}
(\byear{2001}).
\btitle{{On the distribution of the largest eigenvalue in principal component
  analysis.}}
\bjournal{Ann. Stat.}
\bvolume{29}
\bpages{295-327}.
\end{barticle}
\endbibitem

\bibitem[\protect\citeauthoryear{Johnstone and Lu}{2009}]{Johnstone2009}
\begin{barticle}[author]
\bauthor{\bsnm{Johnstone},~\bfnm{I.}\binits{I.}} \AND
  \bauthor{\bsnm{Lu},~\bfnm{Y.}\binits{Y.}}
(\byear{2009}).
\btitle{{On consistency and sparsity for principal components analysis in high
  dimensions}}.
\bjournal{J. Am. Stat. Assoc.}
\bvolume{104}
\bpages{682-693}.
\end{barticle}
\endbibitem

\bibitem[\protect\citeauthoryear{Johnstone and Paul}{2007}]{Johnstone2007}
\begin{bunpublished}[author]
\bauthor{\bsnm{Johnstone},~\bfnm{Iain}\binits{I.}} \AND
  \bauthor{\bsnm{Paul},~\bfnm{Debashis}\binits{D.}}
(\byear{2007}).
\btitle{Sparse principal component analysis for high dimensional data}.
\bnote{Technical report}.
\end{bunpublished}
\endbibitem

\bibitem[\protect\citeauthoryear{Kendall}{1938}]{Kendall1938}
\begin{barticle}[author]
\bauthor{\bsnm{Kendall},~\bfnm{M.}\binits{M.}}
(\byear{1938}).
\btitle{A new measure of rank correlation}.
\bjournal{Biometrika}
\bvolume{30}
\bpages{81-93}.
\end{barticle}
\endbibitem

\bibitem[\protect\citeauthoryear{Lehmann and Casella}{1998}]{Lehmann1998}
\begin{bbook}[author]
\bauthor{\bsnm{Lehmann},~\bfnm{E.}\binits{E.}} \AND
  \bauthor{\bsnm{Casella},~\bfnm{G.}\binits{G.}}
(\byear{1998}).
\btitle{Theory of Point Estimation}.
\bpublisher{Springer, New York}.
\end{bbook}
\endbibitem

\bibitem[\protect\citeauthoryear{Levina and Vershynin}{2012}]{Levina2012}
\begin{barticle}[author]
\bauthor{\bsnm{Levina},~\bfnm{E.}\binits{E.}} \AND
  \bauthor{\bsnm{Vershynin},~\bfnm{R.}\binits{R.}}
(\byear{2012}).
\btitle{Partial estimation of covariance matrices}.
\bjournal{Prob. Theory Rel. Fields}
\bvolume{153}
\bpages{405-419}.
\end{barticle}
\endbibitem

\bibitem[\protect\citeauthoryear{Liu, Lafferty and Wasserman}{2009}]{Liu2009}
\begin{barticle}[author]
\bauthor{\bsnm{Liu},~\bfnm{H.}\binits{H.}},
  \bauthor{\bsnm{Lafferty},~\bfnm{J.}\binits{J.}} \AND
  \bauthor{\bsnm{Wasserman},~\bfnm{L.}\binits{L.}}
(\byear{2009}).
\btitle{The nonparanormal: Semiparametric estimation of high dimensional
  undirected graphs}.
\bjournal{J. Mach. Learn. Res.}
\bvolume{10}
\bpages{2295-2328}.
\end{barticle}
\endbibitem

\bibitem[\protect\citeauthoryear{Liu et~al.}{2012}]{Liu2012}
\begin{barticle}[author]
\bauthor{\bsnm{Liu},~\bfnm{H.}\binits{H.}},
  \bauthor{\bsnm{Han},~\bfnm{F.}\binits{F.}},
  \bauthor{\bsnm{Yuan},~\bfnm{M.}\binits{M.}},
  \bauthor{\bsnm{Lafferty},~\bfnm{J.}\binits{J.}} \AND
  \bauthor{\bsnm{Wasserman},~\bfnm{L.}\binits{L.}}
(\byear{2012}).
\btitle{Highdimensional semiparametric gaussian copula graphical models}.
\bjournal{Ann. Statist.}
\bvolume{40}
\bpages{2293-2326}.
\end{barticle}
\endbibitem

\bibitem[\protect\citeauthoryear{Muir}{1900}]{Muir1900}
\begin{barticle}[author]
\bauthor{\bsnm{Muir},~\bfnm{T.}\binits{T.}}
(\byear{1900}).
\btitle{Determination of the sign of a single term of a determinant}.
\bjournal{Proc. Roy. Soc. Edinb.}
\bvolume{22}
\bpages{441-477}.
\end{barticle}
\endbibitem

\bibitem[\protect\citeauthoryear{Pourahmadi}{2013}]{Pourahmadi2013}
\begin{bbook}[author]
\bauthor{\bsnm{Pourahmadi},~\bfnm{M.}\binits{M.}}
(\byear{2013}).
\btitle{High-Dimensional Covariance Estimation: With High-Dimensional Data}.
\bpublisher{Wiley}.
\end{bbook}
\endbibitem

\bibitem[\protect\citeauthoryear{Saulis and Statulevicius}{1991}]{Saulis1991}
\begin{bbook}[author]
\bauthor{\bsnm{Saulis},~\bfnm{L.}\binits{L.}} \AND
  \bauthor{\bsnm{Statulevicius},~\bfnm{V.}\binits{V.}}
(\byear{1991}).
\btitle{Limit Theorems for Large Deviations}.
\bpublisher{Springer Science \& Business Media}.
\end{bbook}
\endbibitem

\bibitem[\protect\citeauthoryear{Vu and Lei}{2013}]{Vu2013}
\begin{barticle}[author]
\bauthor{\bsnm{Vu},~\bfnm{V.}\binits{V.}} \AND
  \bauthor{\bsnm{Lei},~\bfnm{J.}\binits{J.}}
(\byear{2013}).
\btitle{Minimax sparse principal subspace estimation in high dimensions}.
\bjournal{Ann. Statist.}
\bvolume{41}
\bpages{2905-2947}.
\end{barticle}
\endbibitem

\bibitem[\protect\citeauthoryear{Wegkamp and Zhao}{2013}]{Wegkamp2013}
\begin{bunpublished}[author]
\bauthor{\bsnm{Wegkamp},~\bfnm{M.}\binits{M.}} \AND
  \bauthor{\bsnm{Zhao},~\bfnm{Y.}\binits{Y.}}
(\byear{2013}).
\btitle{Adaptive Estimation of the Copula Correlation Matrix for Semiparametric
  Elliptical Copulas}.
\bnote{Preprint}.
\end{bunpublished}
\endbibitem

\bibitem[\protect\citeauthoryear{Xue and Zou}{2012}]{Xue2012}
\begin{barticle}[author]
\bauthor{\bsnm{Xue},~\bfnm{L.}\binits{L.}} \AND
  \bauthor{\bsnm{Zou},~\bfnm{H.}\binits{H.}}
(\byear{2012}).
\btitle{Regularized rank-based estimation of high-dimensional nonparanormal
  graphical models}.
\bjournal{Ann. Statist.}
\bvolume{40}
\bpages{2541-2571}.
\end{barticle}
\endbibitem

\bibitem[\protect\citeauthoryear{Xue and Zou}{2014}]{Xue2014}
\begin{barticle}[author]
\bauthor{\bsnm{Xue},~\bfnm{L.}\binits{L.}} \AND
  \bauthor{\bsnm{Zou},~\bfnm{H.}\binits{H.}}
(\byear{2014}).
\btitle{Optimal estimation of sparse correlation matrices of semiparametric
  Gaussian copulas}.
\bjournal{Stat. Interface}
\bvolume{7}.
\end{barticle}
\endbibitem

\end{thebibliography}

\end{document}